\documentclass[10pt,reqno]{article}
\usepackage[latin1]{inputenc}   
\usepackage{amsfonts,amsmath,layout}
\usepackage{mathrsfs}
\usepackage{soul}
\usepackage{amssymb,mathabx,amsfonts,amsthm,amscd,stmaryrd,dsfont,esint,upgreek,constants,todonotes}
\DeclareFontFamily{OT1}{pzc}{}
\DeclareFontShape{OT1}{pzc}{m}{it}{<-> s * [1.10] pzcmi7t}{}
\DeclareMathAlphabet{\mathpzc}{OT1}{pzc}{m}{it}

\usepackage{color} 

\definecolor{Red}{cmyk}{0,1,1,0.2}

\usepackage[ntheorem]{empheq}

\newtheorem{theorem}{Theorem}[section]

\newtheorem{lemma}[theorem]{Lemma}
\newtheorem{pro}[theorem]{Proposition}
\newtheorem{definition}[theorem]{Definition}

\theoremstyle{remark}
\newtheorem*{remark}{Remark}

\newcommand{\ep}{\varepsilon}
\newcommand{\ind}{\mathds{1}}
\renewcommand{\d}{\textrm{\,d}}

\newcommand{\R}{\mathbb{R}}

\DeclareMathOperator{\sign}{sign}

\DeclareMathOperator{\ess}{ess}

\DeclareMathOperator{\Hess}{\bold{Hess}}

\begin{document}

\title{ Conservation law and Hamilton-Jacobi equations on a junction: the convex case}

\author{\renewcommand{\thefootnote}{\arabic{footnote}}
  P. Cardaliaguet\footnotemark[1], N. Forcadel\footnotemark[2], T. Girard\footnotemark[3], R. Monneau\footnotemark[1] ~\footnotemark[4]}
\footnotetext[1]{CEREMADE, UMR CNRS 7534, Universit\'e Paris Dauphine-PSL,
Place de Lattre de Tassigny, 75775 Paris Cedex 16, France. }
\footnotetext[2]{INSA Rouen Normandie, Normandie Univ, LMI UR 3226, F-76000 Rouen, France.}
\footnotetext[3]{Institut Denis Poisson, Universit\'e de Tours, Parc Grandmont, 37200 Tours, France.}
\footnotetext[4]{CERMICS, Universit\'e Paris-Est, Ecole des Ponts ParisTech, 6-8 avenue Blaise Pascal, 77455 Marne-la-Vall\'ee Cedex 2, France.}

\maketitle

\begin{abstract}
The goal of this paper is to study the link between the solution to an Hamilton-Jacobi (HJ) equation and the solution to a Scalar Conservation Law (SCL) on a special network. When the equations are posed on the real axis, it is well known that the space derivative of the solution to the Hamilton-Jacobi equation is the solution to the corresponding scalar conservation law. On networks, the situation is more complicated and we show that this result still holds true in the convex case on a 1:1 junction.  The correspondence between solutions to HJ equations and  SCL on a 1:1 junction is done showing the convergence of associated numerical schemes. A second direct proof using semi-algebraic functions is also given.

Here a 1:1 junction is a simple network composed of two edges and one vertex.
In the case of three edges or more, we show that the associated HJ germ is not a $L^1$-dissipative germ, while it is the case for only two edges.

As an important byproduct of our numerical approach, we get a new result on the convergence of numerical schemes for scalar conservation laws on a junction.
For a general desired flux condition which is discretized, we show that the numerical solution with the general flux condition converges to the solution of a SCL problem with an effective flux condition at the junction. Up to our knowledge, in previous works the effective condition was directly implemented in the numerical scheme. In general the effective flux condition differs from the desired one, and is its relaxation, which is very natural from the point of view of Hamilton-Jacobi equations. Here for SCL, this effective flux condition is encoded in a  germ that we characterize at the junction.
\end{abstract}



\section{Introduction}
In one space dimension, it's well known that the space derivative of the viscosity solution to a Hamilton-Jacobi (HJ) equation is the solution to a scalar conservation law (SCL). We refer for example to \cite{Caselles92, colombo-Perrollaz-Sylla} for this kind of results. In this paper, we want to investigate this relation in the case of simple junctions composed of two edges and one vertex (referred later as 1:1 junctions), for which, up to our knowledge, this result is completely open.
Scalar Conservation Laws and Hamilton-Jacobi equations on networks have been largely studied in the last decade. Concerning SCL, the 1:1 case has been studied following many different approaches during the last 20 years (see the two surveys \cite{MISHRASurvey} and \cite{MonotoneGraphAndreianov} and references therein for an overview on the subject). In this paper, we choose to focus mostly on the germ approach (see \cite{andreianov,FMR22}) as it is suitable for the correspondence result. Concerning Hamilton-Jacobi equations on networks, the theory has been largely developed since the pionner works of Achdou, Camilli, Cutr\`{i}, Tchou \cite{ACCT} and Imbert, Monneau, Zidani \cite{IMZ}. We refer in particular to \cite{imbert-monneau}, where a general comparison principle has been developed using PDE tools and a classification of the junction condition has been proposed, to \cite{LS1, LS2} for an extension to the non-convex case and to the monograph \cite{BCbook} for a general review on the topic.

Even if the  theories are now well understood both for scalar conservation laws and Hamilton-Jacobi equations, the relation between these two {theories has} never been addressed on junctions until now. In this paper, we will give an answer for 1:1 junctions and we will also show that the situation is much more complicated when the junction is composed of more than three branches.
The main difficulty comes from the junction condition and we will explain how the junction condition of the HJ equation, namely a flux limiter condition as in \cite{imbert-monneau}, can be interpreted as a condition on a germ, as in \cite{andreianov}.

\subsection{The main result}
The aim of this paper is to make the link between viscosity solutions to Hamilton-Jacobi equations posed on the real line with a discontinuity at the origin and entropy solutions to a suitable conservation law. We consider here the case where the fluxes are convex but the result remains valid in the concave case (just changing the solution $u$ by $-u$).
Namely, we start with the flux-limited viscosity solution $u$, as in \cite{imbert-monneau}, of
\begin{equation}
 \label{eq:HJstrong}
 \left\lbrace \begin{matrix}
 u_t + H_L(u_x) = 0 &\textrm{ if } x < 0 \\
 u_t + H_R(u_x) = 0 &\textrm{ if } x > 0 \\
 u_t + {\bar F_{A}}(u_x(t,0^-), u_x(t,0^+)) = 0 &\textrm{ if } x = 0\\
 u(0,x)=u_0(x) &\textrm{ for } x \in \R
\end{matrix}\right.
\end{equation}
where $u_0$ is a Lipschitz continuous initial condition. For $\alpha=L,R,$ let $a_\alpha<b_\alpha<c_\alpha$.  We make the following assumptions on the Hamiltonians for some $\delta>0$
\begin{equation}\label{ass:H}
\left\{ \begin{array}{c}
\text{For $\alpha=L,R$, the Hamiltonian $H_\alpha$ is of class $C^2$, with $H_\alpha ''\ge \delta>0$,} \\
\text{decreasing on $[a_\alpha,b_\alpha]$ and increasing on $[b_\alpha,c_\alpha]$, with $H_\alpha(a_\alpha)=H_\alpha(c_\alpha)=0$.}
\end{array}
\right.
\end{equation}
 We define the two associated monotone envelopes
$$H_\alpha^+(p)=\left\{\begin{array}{ll}
H_\alpha(b_\alpha) & \quad \mbox{for}\quad p\in [a_\alpha,b_\alpha]\\
H_\alpha(p) & \quad \mbox{for}\quad p\in [b_\alpha,c_\alpha]\\
\end{array}\right.,\quad
H_\alpha^-(p)=\left\{\begin{array}{ll}
H_\alpha(p) & \quad \mbox{for}\quad p\in [a_\alpha,b_\alpha]\\
H_\alpha(b_\alpha) & \quad \mbox{for}\quad p\in [b_\alpha,c_\alpha].\\
\end{array}\right.
$$
Concerning the initial data, we make the following assumption
\begin{equation}\label{eq:u0}
u_0 \textrm{ is Lipschitz continuous {on $\R$} and a.e.} \quad (u_0)_x\in [a_L,c_L] \textrm{ if }x<0\textrm{ and } (u_0)_x\in [a_R,c_R] \textrm{ if }x>0.
\end{equation}
We set
$${H_0:=\max_{\alpha=L,R}\min_p H_\alpha(p)}$$
and for $A\in [H_0, 0]$, we define the effective junction condition ${\bar F_{A}}$ by
\begin{equation}\label{eq:F-A}
{\bar F_{A}(p_L,p_R):=\max\{A,H^+_L(p_L), H^-_R(p_R)\}}
\end{equation}

The goal is then to understand the equation satisfied by
$$\rho:= u_x.$$

\paragraph{Heuristics.}
By \cite{Caselles92, colombo-Perrollaz-Sylla}, we first note that $\rho$ is an entropy solution {to}
\begin{equation}\label{eq:SCL2}
\left\lbrace \begin{matrix}
 \rho_t + H_L(\rho)_x = 0 &\textrm{ if } x < 0 \\
 \rho_t + H_R(\rho)_x = 0 &\textrm{ if } x > 0 \\
 \rho(0,x)=\rho_0(x) &\textrm{ for } x \in \R
\end{matrix}\right.
\end{equation}
where $\rho_0=(u_0)_x$ a.e.. The main difficulty is then to understand what is the appropriate junction condition.
For solutions {to} conservation laws with strongly convex fluxes, we recall the existence of strong traces of $\rho$ at $x=0$ {(see (\ref{eq::g8}) and also \cite{panov-trace})}. We denote by $\rho(t,0^-)$ and $\rho(t,0^+)$ these traces respectively on the left and on the right.
In order to fix a condition at $x=0$ for the scalar conservation law, following the works of \cite{andreianov} and \cite{FMR22, MFR22}, we look for stationary solutions to \eqref{eq:SCL2}, that is  solutions of the form
\begin{equation}\label{eq::t1}
\rho(t,x)=\left\{
\begin{array}{ll}
k_L&{\rm if}\; x<0\\
k_R&{\rm if}\; x>0
\end{array}\quad \quad {\mbox{where $(k_L, k_R)\in Q:=[a_L,c_L]\times[a_R,c_R]$.}}
\right.
\end{equation}
Let us note that, if we set
$$u (t,x)=(k_L x-tH_L (k_L))\ind_{\{x<0\}}+(k_R x-tH_R (k_R))\ind_{\{x>0\}},$$ then $\rho= u_x$ and $u$ is solution to the Hamilton-Jacobi equation \eqref{eq:HJstrong} on $(0+\infty)\times \R\backslash\{0\}$. Since we want $u$ to be continuous at $0$, this implies that the $k_\alpha$ have to satisfy the Rankine-Hugoniot condition
$$H_L(k_L)=H_R(k_R).$$
Moreover, $u$ satisfies the junction condition in \eqref{eq:HJstrong} iff
$$H_R(k_R)=H_L(k_L)=\max(A, H_L^+(k_L), H_R^-(k_R)).$$

\bigskip

Following  \cite{andreianov, FMR22, MFR22}, we then define the germ $\mathcal G_{A}$ as
\begin{equation}\label{eq:GAF0}
   \mathcal G_{A}:=\left\{(k_L,k_R)\in Q,\; H_R(k_R)=H_L(k_L)=\max(A, H_L^+(k_L), H_R^-(k_R))\right\},
\end{equation}
where $Q$ is defined in (\ref{eq::t1}).
We will explain in Proposition \ref{pro:propertiesGAF0} that this germ is {maximal, $L^1$-dissipative and complete}. Hence the following scalar conservation law
\begin{equation}
 \label{eq:SCL-strong}
 \left\lbrace \begin{matrix}
 \rho_t + H_L(\rho)_x = 0 &\textrm{ if } x < 0 \\
 \rho_t + H_R(\rho)_x = 0 &\textrm{ if } x > 0 \\
 (\rho(t,0^-), \rho(t,0^+))\in \mathcal G_{A}&\textrm{a.e} \\
 \rho(0,x)=\rho_0(x) &\textrm{ for } x \in \R

\end{matrix}\right.
\end{equation}
is well-posed.
The first main result of this paper is  the following theorem, which makes rigorous the previous computations.
\begin{theorem}[Viscosity versus entropy solutions:  flux limited conditions]\label{th:main}
Let $u_0$ {satisfy} \eqref{eq:u0} and {let us set}  $\rho_0 =  (u_0)_x$.
Let $H_{L,R}$ satisfying \eqref{ass:H}.
Let $u$ be the unique viscosity solution of \eqref{eq:HJstrong} in the sense of Definition \ref{defi:strongHJ} and $\rho$ be the unique {$\mathcal G_A$-entropy} solution of \eqref{eq:SCL-strong} in the sense of Definition \ref{def:SCLSolTrace}.
Then, in the distributional sense, we have
$$ u_x = \rho.$$
\end{theorem}

We propose two different proofs for this result. The first one {uses numerical schemes} for \eqref{eq:HJstrong} and \eqref{eq:SCL-strong}. More precisely, we propose a numerical approximation for  \eqref{eq:HJstrong} and we consider the numerical derivative of the solution, which gives an appropriate scheme for \eqref{eq:SCL-strong}. Since we have the convergence for the two schemes, we recover the result by passing to the limit. The first advantage of this proof is that it will be generalized in a future work to the non-convex case. The second advantage is that it can be {extended to the important situation of a} more general junction condition, as presented below.

The second approach is a direct proof in which we use a regularization method via the notion of {semi-algebraic} functions (see Section \ref{sec:6}).
\bigskip

\paragraph{General junction conditions.} Up to this point, we only considered a flux-limiter type of junction condition {(with flux-limiter $A$) at the junction point} $x=0$. However it is known that, in the specific setting considered  here, a large class of {coupling conditions} can be equivalently treated as a flux-limiter. Then we  present our result {in this  larger class}. More precisely, we want to consider the general problem

\begin{equation}
 \label{eq:HJ}
 \left\lbrace \begin{matrix}
 u_t + H_L(u_x) = 0 &\textrm{ if } x < 0 \\
 u_t + H_R(u_x) = 0 &\textrm{ if } x > 0 \\
 u_t + F_0(u_x(t,0^-), u_x(t,0^+)) = 0 &\textrm{ if } x = 0\\
 u(0,x)=u_0(x) &\textrm{ for } x \in \R
\end{matrix}\right.
\end{equation}
where the function $F_0: \mathbb{R}^2 \longrightarrow \mathbb{R}$ is called a {\em {desired} coupling condition}  and satisfies the following conditions
\begin{equation}\label{ass:F0}
   \left\{
   \begin{array}{ll}
   \textbf{(Regularity )}&F_0 \textrm{ is Lipschitz continuous and piecewise } C^1(\R^2) \\
   \\
   \textbf{(Monotonicity)} &\begin{array}{l}
    F_0  \textrm{ is {non decreasing} in the first variable}    \\
      \textrm{  and {non increasing} in the second one}
   \end{array}  \\
   \\ \textbf{(Semi-coercivity)} &\displaystyle{\lim_{{\max(0,p_L,-p_R)\to +\infty}}F_0(p_L,p_R)=+\infty}\\
      \\ \textbf{(Boundedness of the solution)} & F_0(a_L,a_R)=F_0(c_L,c_R)=0\\
   \end{array}
   \right.
\end{equation}
Note that the last assumptions will imply that the solutions live in the {box} $Q$ and is naturally satisfied if the junction condition is of the form \eqref{eq:F-A}.
{Moreover notice that it is possible to show a posteriori that the third condition of (\ref{ass:F0}) is not seen by the solution $u$ of (\ref{eq:HJ}) whose gradient $((u_x(t,\cdot))_{|(-\infty,0)},(u_x(t,\cdot))_{|(0,+\infty)})$ stays in the box $Q=[a_L,c_L]\times [a_R,{c_R}]$, {if} the initial data $((\partial_x u_0)_{|(-\infty,0)},(\partial_x u_0)_{|(0,+\infty)})$ does it.

}

It is well-known that, in general, one cannot expect to have a strong viscosity solution for \eqref{eq:HJ}, in the sense that the junction condition is satisfied in the viscosity sense (see Definition \ref{defi:strongHJ} below). Nevertheless, it is always possible to define a weak viscosity solution, meaning that either the equation or the junction condition is satisfied at $x=0$ (see Definition \ref{defi:weakHJ} below).
We are now interested in the corresponding SCL. Formally, we can make the following calculation  with $\rho:=u_x$ {(say with $H(\rho)=0$ at $x=\pm \infty$)}
$$u_t = \partial_t \int_{-\infty}^x \rho \d x =  \int_{-\infty}^x \partial_t \rho \d x
= - \int_{-\infty}^x  {(H(\rho))_x} \d x = -H(\rho).$$
Then {for a solution $u$ of  problem \eqref{eq:HJ}, we expect $\rho:=u_x$ to solve  the scalar conservation law  problem}
\begin{equation}
 \label{eq:SCL}
 \left\lbrace \begin{matrix}
 \rho_t + H_L(\rho)_x = 0 &\textrm{ if } x < 0 \\
 \rho_t + H_R(\rho)_x = 0 &\textrm{ if } x > 0 \\
 H_{L}(\rho(t,0^-)) =H_{R}(\rho(t,0^+))= F_0(\rho(t,0^-), \rho(t,0^+))  &\textrm{ if } x = 0 \\
 \rho(0,x)=\rho_0(x) &\textrm{ for } x \in \R.
\end{matrix}\right.
\end{equation}
However, this problem does not admit a {solution whose traces satisfy the third equation of \eqref{eq:SCL} in general for any given $F_0$ satisfying (\ref{ass:F0})} and one has to relax the junction condition. We recall in Subsection \ref{subsec:2.3} how this problem has to be solved.

We then have the following result.
\begin{theorem}[Viscosity versus entropy solutions:  {desired} conditions]\label{th:main2}
Let $u_0$ {satisfy} \eqref{eq:u0} and denote by $\rho_0 =  (u_0)_x$.
Let $H_{L,R}$ satisfying \eqref{ass:H} and  $F_0$ satisfying \eqref{ass:F0}.
Let $u$ be the unique {weak} viscosity solution to \eqref{eq:HJ} in the sense of Definition \ref{defi:weakHJ} and $\rho$ be the unique $F_0$-admissible solution to \eqref{eq:SCL} in the sense of Definition \ref{defi:weakSCL}.
Then, in the distributional sense, {we have} $$ u_x = \rho.$$
\end{theorem}

This result can be seen as a direct consequence of Theorem \ref{th:main}. Indeed, it is shown in \cite{imbert-monneau} that it possible to construct a flux limiter $A_{F_0}$ depending on $F_0$ (see Lemma \ref{lem:AF0} for this construction) such that the unique weak solution to \eqref{eq:HJ} is in fact the unique strong solution to \eqref{eq:HJstrong} with $A$ replaced by $A_{F_0}$. As the solution of \eqref{eq:SCL} can also be interpreted as the solution of \eqref{eq:SCL-strong}, the result is straightforward. However, we will propose a direct proof of this result. Indeed, the proof using the numerical scheme can be done directly with this type of junction condition. The main point is the following: in the numerical scheme, we will put an approximation of the expected junction condition $F_0$, but at the limit where the space and time steps go to $0$, we will recover the relaxed flux-limited junction condition defined with ${A_{F_0}}$. More precisely, we have the following meta-theorem, which statement is made precise in Theorem \ref{th:CVSCL}.

\begin{theorem}[Numerical approximation for SCL:  {desired} condition]
Let $\rho^\Delta$ (with $\Delta=(\Delta t,\Delta x)$) be the numerical solution of \eqref{eq:SCL} (with the junction condition given by $F_0$). Then, there exists  a flux limiter $A_{F_0}$ depending on $F_0$  such that, as $\Delta$ goes to zero, $\rho^\Delta$ converges to the unique solution to \eqref{eq:SCL-strong} with $A$ replaced by $A_{F_0}$
\end{theorem}

\begin{remark}
Note that this result was already known at the Hamilton-Jacobi level (see \cite{guerand}).
\end{remark}

\begin{remark}
It {is} also possible to consider an even simpler junction constituted of only one edge and one vertex. In that case, our result remains valid with analogous proofs. For instance, the analogue of Theorem \ref{th:main} is precisely the following:
 \begin{theorem}[Viscosity versus entropy solutions: the half line]\label{th::t3}
Let $u_0$ {satisfy} \eqref{eq:u0} on $(0,+\infty)$ and  {let us set}  $\rho_0 =  (u_0)_x$.
Let $H_{R}$ satisfying \eqref{ass:H} and $A\in [\min H_R,0]$.
Let $u$ be the unique viscosity solution {to}
$$\left\lbrace \begin{array}{rlll}
u_t + H_R(u_x) &= 0 &\quad \textrm{if}\quad  &x > 0 \\
u_t + \max\left\{A,H_R^- (u_x(t,0^+))\right\} &= 0 &\quad \textrm{if}\quad  &x = 0\\
u(0,x)&=u_0(x) &\quad \textrm{for} \quad &x \in (0,+\infty)
\end{array}\right.$$
and $\rho$ be the unique $\mathcal G^{1}_A$-entropy solution {to}
$$\left\lbrace \begin{array}{rlll}
 \rho_t + H_R(\rho)_x &= 0 &\quad \textrm{if} \quad &x > 0 \\
 \rho(t,0^+)&\in \mathcal G^1_A  &\quad \textrm{if} \quad &x = 0 \quad \mbox{and for a.e. $t\in (0,+\infty)$}\\
 \rho(0,x)&=\rho_0(x) &\quad \textrm{for} \quad &x \in (0,+\infty).
\end{array}\right.$$
with
$$\mathcal G^1_A:=\left\{k_R\in \R,\quad H_R(k_R)=\max\left\{A,H_R^-(k_R)\right\}\right\}$$
Then, in the distributional sense, we have $$ u_x = \rho.$$
\end{theorem}
The above notion of solution for a scalar conservation law with boundary condition is equivalent to the one given by the standard Bardos-Leroux-Nedelec approach (see \cite{BLN, CR2015}).
\end{remark}

\subsection{Outline}
In Section \ref{sec:3}, we recall the different definitions of {solutions} for \eqref{eq:HJstrong}, \eqref{eq:SCL-strong}, \eqref{eq:HJ} and \eqref{eq:SCL} and we give the link between weak and strong {viscosity} solutions. We also prove {useful} properties on the germ $\mathcal G_{A}$ and we  explain how the flux limiter $A_{F_0}$ is constructed. Section \ref{sec:4} is devoted to the study of the numerical scheme for \eqref{eq:HJ} (Subsection \ref{subsec:4.1}) and \eqref{eq:SCL} (Subsection \ref{subsec:4.1}) while we prove Theorems \ref{th:main} and \ref{th:main2} using these numerical schemes in Section \ref{sec:5}. We propose a direct proof of Theorem \ref{th:main} using regularization with semi-algebraic functions in Section \ref{sec:6}. {Finally Section \ref{s6} is an  appendix where we collect complementary results, which are either new, or not accessible  with full details in the literature. In Subsection \ref{s6a} we give discrete entropy inequalities on a junction, in Subsection \ref{s6b} we give a local compactness result for numerical solutions of conservation laws with strictly convex flux, and in Subsection \ref{s6c} we show that Hamilton-Jacobi germs are not $L^1$-dissipative for $N\ge 3$ branches.}\\

\section{Notions of solution}\label{sec:3}
We begin this section by recalling the definition and some properties of equation \eqref{eq:HJ} in Subsection \ref{subsec:2.2} and of equation \eqref{eq:SCL} in Subsection \ref{subsec:2.3}. Finally, in Subsection \ref{subsec:2.1}, we explain how we construct the flux limiter $A_{F_0}$ from a general condition $F_0$.

\subsection{Definition of weak and strong solutions for Hamilton-Jacobi equations}\label{subsec:2.2}
We begin to recall the notion of weak viscosity solutions to \eqref{eq:HJ}. We consider the set of test functions on the junction ${J_T}:=(0,T)\times \R$:
$${C^1_\wedge(J_T)}:=\{\varphi\in C^0({J_T}),\; \textrm {the restrictions of }\varphi \textrm{ to } (0,T)\times (-\infty,0]\textrm{ and to }  (0,T)\times [0,\infty) \textrm{ are } C^1\}.$$
We also recall the definition of upper and lower semi-continuous envelopes $u^*$ and $u_*$ of a (locally bounded) function $u$ defined on $[0, T )\times \R$,
$$u^*(t, x) = \limsup_{(s,y)\to(t,x)} u(s, y) \quad {\rm and}\quad
u_*(t, x) = \liminf_{(s,y)\to(t,x)} u(s, y).$$

We begin with the notion of strong viscosity solution for which the junction condition is satisfied in a strong sense.

\begin{definition}[Strong viscosity solution]\label{defi:strongHJ}
Let us consider a function $u: \Gamma_T \to \R$.

\noindent {\bf i) (Strong viscosity subsolution)}\\
We say that $u$ is a strong viscosity {subsolution} to \eqref{eq:HJ} if for any point $(t_0,x_0)\in {J_T}$ and any function $\varphi \in {C^1_\wedge(J_T)}$ such that {$u^*-\varphi$} reaches a local maximum at $(t_0,x_0)$
we have
$$\left\{ \begin{matrix}
\varphi_t(t_0,x_0) + H_L(\varphi_x(t_0,x_0)) \le 0 &\textrm{ if } x_0 < 0 \\
\varphi_t(t_0,x_0) + H_R(\varphi_x(t_0,x_0)) \le 0 &\textrm{ if } x_0 > 0 \\
\end{matrix}\right.$$
when $x_0\ne 0$ and
$$\varphi_t(t_0,x_0) + F_0(\varphi_x(t_0,x_0^-),\varphi_x(t_0,x_0^+) ) \le 0$$
when $x_0=0$.
We call $u$ a strong $F_0$-{subsolution}.\\
\noindent {\bf ii) (Strong viscosity supersolution)}\\
We say that $u$ is a strong viscosity {supersolution} to \eqref{eq:HJ} if for any point $(t_0,x_0)\in {J_T}$ and any function $\varphi \in {C^1_\wedge(J_T)}$ such that ${u_*-\varphi}$ reaches a local minimum at $(t_0,x_0)$
we have
$$\left\{ \begin{matrix}
\varphi_t(t_0,x_0) + H_L(\varphi_x(t_0,x_0)) \ge 0 &\textrm{ if } x_0 < 0 \\
\varphi_t(t_0,x_0) + H_R(\varphi_x(t_0,x_0)) \ge 0 &\textrm{ if } x_0 > 0 \\
\end{matrix}\right.$$
when $x_0\ne 0$ and
$$\varphi_t(t_0,x_0) + F_0(\varphi_x(t_0,x_0^-),\varphi_x(t_0,x_0^+) )  \ge 0$$
when $x_0=0$.
We call $u$ a strong $F_0$-{supersolution}.\\
\noindent {\bf ii) (Strong viscosity solution)}\\
We say that $u$ is a strong viscosity solution to \eqref{eq:HJ}, {if $u$ is a strong viscosity {subsolution} to (\ref{eq:HJ}), and $u$} is a strong viscosity {supersolution} to (\ref{eq:HJ}). We call $u$ a strong $F_0$-solution.
\end{definition}
A first result of Imbert, Monneau \cite{imbert-monneau} is that when the junction condition is of the form ${\bar F_A}$ in (\ref{eq:F-A}), then the junction condition is satisfied strongly as in the previous definition. Nevertheless, this is not true for general junction condition and one has to consider weak viscosity solutions for which either the junction condition or the equation is satisfied at $x=0$.

\begin{definition}[Weak viscosity solution]\label{defi:weakHJ}
Let us consider a function $u: \Gamma_T \to \R$

\noindent {\bf i) (Weak viscosity subsolution)}\\
We say that $u$ is a weak viscosity {subsolution} to \eqref{eq:HJ} if for any point $(t_0,x_0)\in {J_T}$ and any function $\varphi \in {C^1_\wedge(J_T)}$ such that {$u^*-\varphi$} reaches a local maximum at $(t_0,x_0)$
we have
$$\left\{ \begin{matrix}
\varphi_t(t_0,x_0) + H_L(\varphi_x(t_0,x_0)) \le 0 &\textrm{ if } x_0 < 0 \\
\varphi_t(t_0,x_0) + H_R(\varphi_x(t_0,x_0)) \le 0 &\textrm{ if } x_0 > 0 \\
\end{matrix}\right.$$
when $x_0\ne 0$ and
$$
  \varphi_t(t_0,x_0) + H_L(\varphi_x(t_0,x_0^-)) \le 0\quad \mbox{or}\quad \varphi_t(t_0,x_0) + H_R(\varphi_x(t_0,x_0^+)) \le 0
  $$
  $$
  \mbox{or}\quad \varphi_t(t_0,x_0) + F_0(\varphi_x(t_0,x_0^-),\varphi_x(t_0,x_0^+) ) \le 0
$$
when $x_0=0$. We call $u$ a weak $F_0$-subsolution.\\
\noindent {\bf ii) (Weak viscosity supersolution)}\\
We say that $u$ is a weak viscosity {supersolution} to \eqref{eq:HJ} if for any point $(t_0,x_0)\in {J_T}$ and any function $\varphi \in {C^1_\wedge(J_T)}$ such that {$u_*-\varphi$} reaches a local minimum at $(t_0,x_0)$
we have
$$\left\{ \begin{matrix}
\varphi_t(t_0,x_0) + H_L(\varphi_x(t_0,x_0)) \ge 0 &\textrm{ if } x_0 < 0 \\
\varphi_t(t_0,x_0) + H_R(\varphi_x(t_0,x_0)) \ge 0 &\textrm{ if } x_0 > 0 \\
\end{matrix}\right.$$
when $x_0\ne 0$ and
$$
  \varphi_t(t_0,x_0) + H_L(\varphi_x(t_0,x_0^-)) \ge 0\quad \mbox{or}\quad \varphi_t(t_0,x_0) + H_R(\varphi_x(t_0,x_0^+)) \ge 0 
  $$
  $$\mbox{or}\quad \varphi_t(t_0,x_0) + F_0(\varphi_x(t_0,x_0^-),\varphi_x(t_0,x_0^+) ) \ge 0
$$
when $x_0=0$. We call $u$ a weak $F_0$-supersolution.\\
\noindent {\bf iii) (Weak viscosity solution)}\\
We say that a locally bounded function $u$ is a weak viscosity solution to (\ref{eq:HJ}), {if $u$ is a weak viscosity {subsolution} to (\ref{eq:HJ}), and $u$} is a weak viscosity {supersolution} to (\ref{eq:HJ}). We call $u$ a weak $F_0$-solution.
\end{definition}

An important result of Imbert, Monneau \cite{imbert-monneau} is that it is possible to relax the junction condition  in order to make the solution satisfy the junction condition strongly. We refer to Subsection \ref{subsec:2.1} for the construction of the relaxation and to Theorem \ref{th:equivHJ} for the precise result. Let us also mention that the existence and uniqueness (using a comparison principle) of the solutions of \eqref{eq:HJstrong} and \eqref{eq:HJ} is also proven in \cite{imbert-monneau}. In particular \eqref{eq:HJ} admits a strong solution if and only if $F_0$ is of the form $\bar F_A$ for some $A\in [H_0,0]$.

\subsection{Definition of solution for conservation law}\label{subsec:2.3}
We first recall that any solution to a Scalar Conservation Law {for $x\in (0,+\infty)$} with strongly convex flux has a strong trace at $x=0$ (see Panov \cite[Theorem 1.1]{panov-trace}).
For any {function $f : (0,T)\times \mathbb{R} \longrightarrow \mathbb{R}$,  we denote by $\gamma_{L,R}f$} the strong  left and right traces {of $f$} at $x=0$ when they exist. For instance for the left trace, this means that
\begin{equation}\label{eq::g8}
\ess \lim_{x\to  0^-}  \int_0^T  \left| f(t,x) - {\gamma_L}f(t) \right|  \d t = 0.
\end{equation}

Here we present the notion of solution we will {consider} for \eqref{eq:SCL-strong}. We consider an effective junction condition $F_A$ as defined in \eqref{eq:F-A} and we recall that the corresponding germ $\mathcal G_A$ is given by \eqref{eq:GAF0}.

\begin{definition}[Strong entropy solution]\label{def:SCLSolTrace}
Let $u_0$ satisfying \eqref{eq:u0} and denote by $\rho_0 =  (u_0)_x$.
We say that $\rho \in L^{\infty}((0,T)\times \mathbb{R})$ is a ``strong" $\mathcal G_A$-entropy solution to \eqref{eq:SCL-strong} if
\begin{enumerate}
\item $\rho$ is a weak solution to
 \begin{equation}\nonumber\left\lbrace \begin{matrix}
 \rho_t + H_L(\rho)_x = 0 &\textrm{ if } x < 0 \\
 \rho_t + H_R(\rho)_x = 0 &\textrm{ if } x > 0. \\
 \end{matrix}\right.
 \end{equation}
\item For any $\phi_{L} \in {C}^{\infty}_c([0,T)\times \mathbb{R}^{-})$ (resp. $\phi_{R} \in {C}^{\infty}_c([0,T)\times \mathbb{R}^{+})$) that is non-negative,
for any $k_L \in [a_L,c_L]$ (resp. $k_R \in [a_R,c_R]$) the following entropy inequalities hold
\begin{equation}\nonumber
\iint_{(0,T)\times \mathbb{R}^{-}} |\rho-k_L| (\phi_L)_t + \sign(\rho-k_L)\left[H_{L}(\rho) - H_{L}(k_L)\right] (\phi_L)_x
+ \int_\mathbb{R^-} |\rho_0(x) -k_L | \phi_L(0,x) \d x \geq 0
\end{equation}
\bigg(resp.
\begin{equation}\nonumber
\iint_{(0,T)\times \mathbb{R}^{+}} |\rho-k_R| (\phi_R)_t + \sign(\rho-k_R)\left[H_{R}(\rho) - H_{R}(k_R)\right] (\phi_R)_x + \int_\mathbb{R^+} |\rho_0(x) -k_R | \phi_R(0,x) \d x \geq 0\bigg).
\end{equation}
\item The strong traces satisfy the germ condition
\begin{equation}\nonumber
(\gamma_L\rho(t), \gamma_R\rho(t))\in \mathcal G_{A}\quad {\textrm{for a.e. $t\in (0,T)$}}.
\end{equation}
\end{enumerate}
\end{definition}

As proved in \cite{andreianov}, this notion of solution grants existence and uniqueness as soon as the germ $\mathcal G_A$ is $L^1$ dissipative{, maximal and complete}.
We begin by recalling the notion of {$L^1$-dissipativity,  maximality and completeness of a germ.

\begin{definition}[Germ and properties]$\mbox{ }$\\
\noindent {\bf i) (germ)}\\
We say that a set $\mathcal G \subset \R^2$ is a germ if any element of $\mathcal G$ satisfies the Rankine-Hugoniot condition, i.e.
$$H_L(k_L)=H_R(k_R)\quad \forall k=(k_L,k_R)\in \mathcal G.$$
\noindent {\bf ii) ($L^1$-dissipative germ)}\\
We say that a germ $\mathcal G$ is $L^1$-dissipative if for any $k=(k_L,k_R),\hat k=(\hat k_L,\hat k_R)\in \mathcal G$, we have
$$\text{sgn}(k_L-\hat k_L) (H_L(k_L)-H_L(\hat k_L)) \geq \text{sgn}(k_R-\hat k_R) (H_R(k_R)-H_R(\hat k_R)).
$$
\noindent {\bf iii) (maximal $L^1$-dissipative germ)}\\
{A} $L^1$-dissipative germ $\mathcal G$  is called
maximal if there is no  $L^1$-dissipative germ $\bar {\mathcal G}$ having $\mathcal G$ as a strict
subset.\medskip

\noindent {\bf iv) (complete $L^1$-dissipative germ)}\\
{A} $L^1$-dissipative germ $\mathcal G$  is called complete (on the box $Q$), if for every $\hat k=(\hat k_L,\hat k_R)\in Q$, there exists a strong  $\mathcal G_A$-entropy solution of \eqref{eq:SCL-strong}, with initial data $\rho_0= \hat k_L 1_{(-\infty,0)} + \hat k_R 1_{(0,+\infty)}$.
\end{definition}}

We then have the following theorem.
\begin{theorem}[Existence and uniqueness for \eqref{eq:SCL-strong}, \cite{andreianov}]\label{th:eqSolEntropy}
  Let $\rho_0$ be an initial data satisfying $\rho_0((-\infty,0))\times \rho_0((0,+\infty)) \subset Q$.
  \begin{itemize}
  \item[$(i)$]
If the germ $\mathcal G_A$ is $L^1$-dissipative and maximal, there exists at most one solution to \eqref{eq:SCL-strong} in the sense of Definition \ref{def:SCLSolTrace}.
 \item[$(ii)$] Furthermore, if the germ $\mathcal G_A$ is also complete (on the box $Q$), then there exists a unique solution to \eqref{eq:SCL-strong} in the sense of Definition \ref{def:SCLSolTrace}. 
 \end{itemize}
\end{theorem}

In order to apply this result to \eqref{eq:SCL-strong}, it remains to show that the germ $\mathcal G_{A}$ defined in \eqref{eq:GAF0} is {$L^1$-dissipative maximal and complete.
\begin{pro}[$\mathcal G_{A}$ is $L^1$-dissipative, maximal and complete]\label{pro:propertiesGAF0}  Let $A\in [H_0,0]$. We recall that  $F_A(k_L,k_R) = \max\{ A, H^-_L(k_L), H^+_R(k_R) \}$. Then, the set $\mathcal G_A$ defined by
\begin{align}
   \mathcal G_{A}&= \left\{ (k_L,k_R)\in \R^2,\;H_R(k_R)=H_L(k_L)=\bar F_A(k_L,k_R) \right\}\nonumber\\
   &=\left\{\begin{array}{ll}
  & (k_L,k_R)\in \R^2,\;H_R(k_R)=H_L(k_L)\ge A \textrm{ and }\\
   &[\textrm{either } H_R(k_R)=A, \textrm{or } H_R(k_R)=H_R^-(k_R), \textrm{or } H_L(k_L)=H_L^+(k_L)] \end{array} \right\}\label{eq:406}
\end{align}
is a maximal and complete $L^1$-dissipative germ.
\end{pro}}
\begin{remark}
This Definition of the germ $\mathcal{G}_A$ is close to the definition of viscosity solution for Hamilton-Jacobi's equations.
One can also relate this germ to the classical flux limited notion of solution for scalar conservation law with applications to traffic (see \cite{ColomboGoatin2007} and \cite{AGS2010}).

This germ is also the unique maximal $L^1$-dissipative germ containing $(\bar p_L,\bar p_R)$ where $(\bar p_L,\bar p_R)$ is the unique couple such that $A=H_R^+(\bar p_R)=H_L^-(\bar p_L)$. This corresponds to the {so called} $(A,B)$-connection if one takes $(A,B) = (\bar p_L,\bar p_R)$ (see \cite{AdimurthiGowda2005}). {Notice also that contrarily to \cite{AdimurthiGowda2005} and \cite{andreianov}, we do not need any crossing condition to be satisfied}.\\
Finally, we can also link this definition with the monotone graph approach introduced in \cite{MonotoneGraphAndreianov}. If one takes $\Gamma_0 := \{ (p_L,p_R,F_0(p_L,p_R), F_0(p_L,p_R)), \; \; (p_L,p_R) \in \mathbb{R}^2 \}$ then the projected maximal monotone graph is $\Gamma = \{ (p_L,p_R,{\bar F_A(p_L,p_R), \bar F_A(p_L,p_R)}), \; \;  (p_L,p_R) \in \mathcal{G}_{A_{F_0}} \}$.
\end{remark}

\begin{proof}[Proof of Proposition \ref{pro:propertiesGAF0}]
We begin to prove that the germ is $L^1$-dissipative. Let $k=(k_L,k_R),\hat k = (\hat k_L, \hat k_R) \in \mathcal G_A$. We have to show that
\begin{equation}\label{eq:constistencyGerm}
\text{sgn}(k_L-\hat k_L) (H_L(k_L)-H_L(\hat k_L)) \geq \text{sgn}(k_R-\hat k_R) (H_R(k_R)-H_R(\hat k_R)).
\end{equation}
The result is obvious if $H_L(k_L)-H_L(\hat k_L)= H_R(k_R)-H_R(\hat k_R)=0$ or if $H_L(k_L)-H_L(\hat k_L)= H_R(k_R)-H_R(\hat k_R)>0$ and $k_L>\hat k_L$. Let us now assume to fix the ideas that $H_L(k_L)-H_L(\hat k_L)= H_R(k_R)-H_R(\hat k_R)>0$ and $k_L<\hat k_L$ (the case $H_L(k_L)-H_L(\hat k_L)= H_R(k_R)-H_R(\hat k_R)<0$ and $k_L>\hat k_L$ { is obtained exchanging $k$ and $\hat k$}). We need to check that $k_R<\hat k_R$. Note that
$$
H_L(k_L)>H_L(\hat k_L)\ge H_L^+(\hat k_L)\ge H_L^+(k_L).
$$
Since $H_R(k_R)=H_L(k_L)> H_L(\hat k_L)\ge A$, and since $k\in \mathcal G_A$,  we necessarily have  $H_R^-(k_R)=H_R(k_R)$. Therefore
$$
H_R^-(k_R)=H_R(k_R)>H_R(\hat k_R)\ge H_R^-(\hat k_R),
$$
which implies that $\hat k_R>k_R$. This proves \eqref{eq:constistencyGerm} and the $L^1$ dissipativity of $\mathcal G_A$.
\bigskip

To prove the maximality of $\mathcal G$, let us now fix some $k\in \R^2$ such that $H_L(k_L)=H_R(k_R)$ and assume that \eqref{eq:constistencyGerm} holds for any $\hat k\in \mathcal G_A$. We have to check that $k\in \mathcal G_A$. We first check that $H_L(k_L)=H_R(k_R)\ge A$.
By contradiction, assume that  $H_L(k_L)=H_R(k_R)< A$.
We take $\hat k$ such that $\hat k_L$ is the smallest element in $(H_L)^{-1}(\{ A\})$ and $\hat k_R$ the largest in $(H_R)^{-1}(A)$. Then $\hat k\in \mathcal G_A$,  $\hat k_L<k_L$, $\hat k_R>k_R$ and $H_L(k_L)< A= H_L(\hat k_L)$ {and similarly $H_R(k_R)< A= H_R(\hat k_R)$}, which contradicts \eqref{eq:constistencyGerm}. So $H_L(k_L)=H_R(k_R)\ge A$.\medskip

We now prove that
\begin{equation}\label{eq:either}
    H_L(k_L)=A  \textrm{ or } H_R(k_R)=H_R^-(k_R)  \textrm{ or } H_L(k_L)=H_L^+(k_L).
\end{equation}
By contradiction, assume that
$$
H_L(k_L)> A, \; H_L^+(k_L)<H_L(k_L)\;  \text{and}\; H_R^-(k_R)<H_R(k_R).
$$
Let us choose $\hat k\in \mathcal G_A$ such that $H_L(\hat k_L)= H_L^+(\hat k_L)=A$ and $H_R(\hat k_R)= H_R^-(\hat k_R)= A$. Then, as $H_L$ and $H_R$ are convex and as $H_L^+(k_L)<H_L(k_L)$ and $H_R^-(k_R)<H_R(k_R)$, we have
$$
H_L^+(k_L)=\min H_L\le  A = H_L^+(\hat k_L),
$$
which implies that $\hat k_L> k_L$ (equality cannot hold because $H_L(k_L)> {A=H_L(\hat k_L)}$) while
$$
H_R^-(k_R) =\min H_R \le  A = H_R^-(\hat k_R),
$$
which implies that $\hat k_R<k_R$ {(because $H_R(k_R)=H_L(k_L)> A=H_R(\hat k_R)$)}. This yields a contradiction with \eqref{eq:constistencyGerm}. Therefore $k$ satisfies \eqref{eq:either}  and belongs to $\mathcal G_A$. This shows the maximality of $\mathcal G_A$.\\
The proof of the completeness of the germ $\mathcal G_A$ is {postponed} to Lemma \ref{lem::t4}, where we show the existence of a solution using the convergence of the numerical scheme introduced in Subsection \ref{subsec:4.2}.
\end{proof}

\begin{remark}
In the case of a junction with $N\ge 3$ branches, it is possible to show that the Hamilton-Jacobi germ is never $L^1$-dissipative (except in the special case where the limiter $A=0$ which corresponds to no flux at the junction point). See Lemma \ref{lem::g5}.
\end{remark}

We now present an important result telling that the gem $\mathcal G_A$ is generated by a set of three points :  
\begin{equation}\label{eq:EA}
\mathcal E_A:=\{(a_L,a_R), (c_L,c_R), (\bar p_L^A, \bar p_R^A)\},
\end{equation}
where $(\bar p_L^A, \bar p_R^A)$ is such that
$$H_L(\bar p_L^A)=H_L^-(\bar p_L^A)=A=H_R^+(\bar p_R^A)=H_R(\bar p_R^A).$$
This fact was already mentioned in \cite{AdimurthiGowda2005}.
\begin{lemma}[$\mathcal E_A$ generates $\mathcal G_A$ {on  $Q$}]\label{lem:generate}
{Assume that $A\in [H_0,0]$. Then the set $\mathcal E_A$ generates $\mathcal G_A$ on the box $Q$}: namely, for any $(k_L,k_R)\in Q$,
$$
\Bigl( \; q_L(k_L, \bar k_L)-q_R(k_R,\bar k_R) \geq 0 \qquad \forall  (\bar k_L,\bar k_R)\in \mathcal E_A\;\Bigr) \qquad \Longrightarrow \qquad (k_L,k_R)\in \mathcal G_A,
$$
where, for $\alpha =L,R$, $q_\alpha$ are the entropy fluxes defined by
$$q_\alpha(q,p) = \mbox{sign}(q-p)(H_\alpha(q)-H_\alpha(p)).$$
\end{lemma}

\begin{proof}
We choose $(k_L,k_R)\in Q$ and we will test it with the elements $(\bar k_L,\bar k_R)\in \mathcal E_A$
using the dissipation condition in order to show that $(k_L,k_R)\in {\mathcal G}_A$.\\
\noindent {\bf Step 1: recovering Rankine-Hugoniot condition}\\
We choose $(\bar k_L,\bar k_R)=(a_L,a_R)$. We then have
$$0\le q_L(k_L, \bar k_L)-q_R(k_R,\bar k_R)= \mbox{sign}(k_L-a_L)H_L(k_L)- \mbox{sign}(k_R-a_R)H_R(k_R).$$
Since $k_L\ge a_L$ and $k_R\ge a_R$, we recover that $H_L(k_L)\ge H_R(k_R)$. In the same way, taking $(\bar k_L,\bar k_R)=(c_L,c_R)$, we get
$H_L(k_L)\le H_R(k_R)$, which implies {that}
$$H_L(k_L)= H_R(k_R).$$\medskip

\noindent {\bf Step 2: $H_L(k_L)=H_R(k_R)\ge A$}\\
We choose $(\bar k_L,\bar k_R)=(\bar p_L^A, \bar p_R^A)$ and by contradiction, we assume that
$$H_L(k_L)=H_R(k_R)<A=H_L(\bar k_L)=H_L(\bar k_R).$$
Since
$${H_L^-(\bar k_L)=H_L(\bar k_L)}>H_L(k_L)\ge H_L^-(k_L),$$
we deduce that $\bar k_L<k_L$. In the same way, we get $\bar k_R>k_R$.
Using that
$$0\le q_L(k_L, \bar k_L)-q_R(k_R,\bar k_R)= \mbox{sign}(k_L-\bar k_L)(H_L(k_L)-A)- \mbox{sign}(k_R-\bar k_R)(H_R(k_R)-A)<0$$
we get a contradiction.\medskip

\noindent {\bf Step 3: $H_L(k_L)=H_R(k_R)={\bar F_A}(k_L,k_R)$}\\
We choose $(\bar k_L,\bar k_R)=(\bar p_L^A, \bar p_R^A)$ and by contradiction, we assume that
$${H_R(k_R)=}H_L(k_L)>A\quad{\rm and}\quad H_L(k_L)>H_L^+(k_L)\quad{\rm and}\quad H_R(k_R)>H_R^-(k_R).$$
Using that $H_L(k_L)=H_L^-(k_L)>A=H_L^-(\bar k_L)$, we deduce that $k_L<\bar k_L$. In the same way, we have $k_R>\bar k_R$.
This implies that
$$0\le q_L(k_L, \bar k_L)-q_R(k_R,\bar k_R)= \mbox{sign}(k_L-\bar k_L)(H_L(k_L)-A)- \mbox{sign}(k_R-\bar k_R)(H_R(k_R)-A)<0$$
which is a contradiction.
\end{proof}

\bigskip

\paragraph{General junction condition for SCL.} We now explain how the Scalar Conservation Law \eqref{eq:SCL} should be treated.
 Following the approach of \cite{CancesAndreianov2015}, the idea to understand this problem is to study two half-space problems for two given Dirichlet boundary condition $(k_L,k_R)$,

\begin{align}
&\left\lbrace \begin{matrix}
 \rho_t + H_L(\rho)_x = 0 &\textrm{ if } x < 0 \\
   \rho(t,0^-) = k_L(t) \\
 \rho(0,x)=\rho_0(x) &\textrm{ for } x < 0
\end{matrix}\right.
&\left\lbrace \begin{matrix}
 \rho_t + H_R(\rho)_x = 0 &\textrm{ if } x > 0 \\
   \rho(t,0^+) = k_R(t) \\
 \rho(0,x)=\rho_0(x) &\textrm{ for } x > 0
\end{matrix}\right.  \label{eq:SCL-weak}
\end{align}
where the couple of boundary conditions $(k_L(t),k_R(t))$ satisfies the following transmission condition
\begin{equation}\label{eq:TransmissionCondition}
H_L(\rho(t,0-))  = H_R(\rho(t,0+)) =F_0(k_L(t),k_R(t))  .
\end{equation}
Moreover, the Dirichlet boundary conditions {in (\ref{eq:SCL-weak})} have to be understood in the sense of Bardos-Leroux-Nedelec (see \cite{BLN}), i.e.
\begin{align*}\label{eq:BLNsense}
H_L(\rho(t,0-)) = g^{H_L}(\rho(t,0-),k_L(t)), \qquad  H_R(\rho(t,0+)) = g^{H_R}(k_R(t),\rho(t,0+)) \; \; \textrm{ for a.e. } t \in (0,T)
\end{align*}
where for a general Hamiltonian $H$, $g^H$ is the Godunov flux defined by
$$g^{H}(p_1,p_2)=\left\{\begin{array}{lll}
\min_{p\in [p_1,p_2]}H(p)&{\rm if}&p_1\le p_2\\
\max_{p\in [p_2,p_1]}H(p)&{\rm if}&p_2\le p_1.
\end{array}
\right.
$$

We say that a solution to \eqref{eq:SCL-weak}-\eqref{eq:TransmissionCondition} is a $F_0$-admissible solution to \eqref{eq:SCL}.\\
In our specific setting, due to the monotonicity of $F_0$, for any couple $(\rho^L,\rho^R) \in \mathbb{R}^2$ verifying $H_L(\rho^L) = H_R(\rho^R)$ there exists a unique value $F(\rho^L,\rho^R) \in \mathbb{R}$ such that there exists $(k_L,k_R) \in \mathbb{R}^2$ satisfying $F_0(k_L,k_R) = F(\rho^L,\rho^R)$ and \eqref{eq:TransmissionCondition} with $\rho(t,0^-) =\rho^L$ and $\rho(t,0^+) =\rho^R$ (see \cite{BCD2016} and \cite{CancesAndreianov2015}).
Moreover, one can show {(the reader can try to check it directly, but this result will be addressed in a much more generality in a future work)} that
$$F(\rho(t,0^-),\rho(t,0^+)) = \max(A_{F_0}, H_L^+(\rho(t,0^-)), H_R^-(\rho(t,0^+))),$$
where $A_{F_0}$ is constructed in Lemma \ref{lem:AF0} below.
Then, solving \eqref{eq:TransmissionCondition} rewrites as
$$H_R(\rho(t,0^-))=H_L(\rho(t,0^+))= {\bar F_{A_{F_0}}}(\rho(t,0^-),\rho(t,0^+)).$$
which is exactly the junction condition that $\rho$ must satisfy in \eqref{eq:SCL-strong}.\\
We then define  the solution of \eqref{eq:SCL} as follow.
\begin{definition}[Definition of solutions to \eqref{eq:SCL}]\label{defi:weakSCL}
We say that $\rho\in L^\infty((0,T)\times \R)$ is a $F_0$-admissible solution to \eqref{eq:SCL} if $\rho$ is a $\mathcal G_{A_{F_0}}$-entropy  solution to \eqref{eq:SCL-strong} with $A_{F_0}$ defined in Lemma \ref{lem:AF0}.
\end{definition}

\subsection{Construction of the flux limiter $A_{F_0}$}\label{subsec:2.1}
In this section, given a desired junction condition $F_0$ satisfying \eqref{ass:F0}, we want to define the relaxed junction condition such that the weak viscosity solution to \eqref{eq:HJ} satisfies the relaxed junction condition strongly. This junction condition is of the form
${\bar F_{A_{F_0}}}$ (see \eqref{eq:F-A}), where the constant $A_{F_0}$ depends on $F_0$ and is defined as the unique constant such that there exists $\bar p=(\bar p_l, \bar p_R)$ such that
$$A_{F_0}=F_0(\bar p)=H_R^+(\bar p_R)=H_L^-(\bar p_L).$$
More precisely, we have the following lemma (see also \cite[Lemma 2.13]{imbert-monneau}):
\begin{lemma}[Definition of the flux limiter $A_{F_0}$]\label{lem:AF0}
   Let $F_0$ and $H_\alpha$, $\alpha=L,R$ satisfy respectively \eqref{ass:F0} and \eqref{ass:H}. We denote by
   $$H_0:=\max_{\alpha=L,R}\min H_\alpha(p)=\max(H_L(b_L), H_R(b_R)),$$
where we recall that $b_\alpha$ is the  point of minimum of $H_\alpha$.

{Let $\bar b_R$ be the maximal $p$ such that $H_R(p)=H_0$, and $\bar b_L$ be the minimal $p$ such that $H_L(p)=H_0$.}
If $F_0(\bar b_L, \bar b_R)< H_0$, we set $A_{F_0}=H_0$.
If $F_0(\bar b_L, \bar b_R)\ge H_0$, then we define the set
   $$\Lambda:=\{\lambda\in \R, \exists \bar p=(\bar p_L,\bar p_R) \; {\rm s.t.}\; \lambda=F_0(\bar p)=H_R^+(\bar p_R)=H_L^-(\bar p_L)\}.$$
   Then $\Lambda $ is non-empty and is reduced to a singleton. We denote by $A_{F_0}$ the unique constant such that
   $$\Lambda=\{A_{F_0}\}.$$
   {Moreover, if $F_0=\bar F_A$ with $A\in [H_0,0]$, then $A_{F_0}=A$.}
\end{lemma}

\begin{proof}
{\bf Step 1: $\Lambda$ is non empty.}
Given $\lambda > H_0$, we define $p_\alpha^\lambda$ such that
$$H_L^-(p_L^\lambda)=H_R^+(p_R^\lambda)=\lambda.$$
{For $\lambda=H_0$, we set
$$p^{H_0}_\alpha:=\lim_{\lambda\to (H_0)^+} p_\alpha^\lambda$$
which satisfies $p_R^{H_0}=\bar  b_R$ and $p_L^{H_0}=\bar b_L$.}
In particular, {the map $\lambda\mapsto p_R^\lambda$} is continuous and increasing, while {the map $\lambda\mapsto p_L^\lambda$} is continuous and decreasing. Since $F_0$ is non-decreasing in the first variable and non-increasing in the second one, the map $\lambda\mapsto F_0(p_L^\lambda,p_R^\lambda)$ is non-increasing.

We then define the application $K:\lambda\mapsto F_0(p_L^\lambda,p_R^\lambda)-\lambda$.  {When $F_0(p_L^{H_0},p_R^{H_0})= F_0(\bar b_L, \bar b_R)\ge H_0$}, we get that
$K(H_0)\ge 0$. Using the fact of $\lambda\mapsto F_0(p_L^\lambda,p_R^\lambda)$ is non-increasing, we also have
$$K(\lambda)\le {F_0(\bar b_L, \bar b_R)}-\lambda$$
and so for $\lambda $ large enough, we have $K(\lambda)<0$. By continuity, this implies that there exists $\bar \lambda \ge H_0$ such that $K(\bar \lambda)=0$. We set $\bar p=(p_L^{\bar \lambda}, p_R^{\bar \lambda})$ and we get
$$F_0(\bar p)=\bar \lambda=H_L^-(\bar p_L)=H_R^+(\bar p_R),$$
i.e. ${\bar \lambda}\in \Lambda$.
\medskip

{\bf Step 2: $\Lambda$ is reduced to a singleton.}
Assume that there exists $\bar \lambda_1, \bar \lambda_2\in \Lambda$ such that $\bar \lambda_1>\bar \lambda_2$. Hence, there exists $p^i_R$ and $p^i_L$ such that
$$\bar \lambda_1=F_0(p^1_L,p^1_R)=H_L^-(p^1_L)=H_R^+(p_R^1)>\bar \lambda_2=F_0(p^2_L,p^2_R)=H_L^-(p^2_L)=H_R^+(p_R^2)$$
In particular,  we have $p^1_L<p^2_L$ and $p^1_R>p^2_R$. By monotonicity of $F_0$, this implies that
$$F_0(p^1_L,p^1_R)\le F_0(p^2_L, p^2_R),$$
which is a contradiction.\end{proof}

As explained before, the solution of \eqref{eq:HJ} is satisfied in a weak sense for general $F_0$. Nevertheless, it is possible to relax the junction condition  in order to make the solution satisfy the junction condition strongly. More precisely, we have the following theorem, given in \cite[Proposition 2.12]{imbert-monneau}.

\begin{theorem}[General junction conditions reduce to flux limited ones]\label{th:equivHJ}
   Assume that $H_L$ and $H_R$ satisfy \eqref{ass:H} and that $F_0$ satisfies \eqref{ass:F0}. Then $u$ is a continuous weak viscosity solution to \eqref{eq:HJ}, if and only if $u$ is a strong viscosity solution to \eqref{eq:HJstrong} with  {$\bar F_A$ for $A:=A_{F_0}$} defined above in Lemma~\ref{lem:AF0}.
\end{theorem}

\section{Numerical schemes}\label{sec:4}
\subsection{Numerical scheme for the Hamilton-Jacobi equation \eqref{eq:HJ}}\label{subsec:4.1}
In this subsection, we describe the numerical scheme used to solve the Hamilton-Jacobi equation \eqref{eq:HJ}. Given a time step $\Delta t>0$ and a space step $\Delta x>0$, we consider the discrete time $t_n=n\Delta t$ for $n\in \mathbb N$ and the discrete point $x_{j}=j\Delta x$ for $j\in \mathbb Z$. We denote by $u^n_j$ the numerical approximation of $u(t_n,x_j)$. In order to discretize \eqref{eq:HJ}, we will use a Godunov approximation. More precisely, we introduce the following Godunov numerical Hamiltonians, for $\alpha=L,R$
$$g^{H_\alpha}(p^-,p^+)=\left\{\begin{array}{lll}
\min_{p\in [p^-,p^+]}H_\alpha(p)&{\rm if}&p^-\le p^+\\
\max_{p\in [p^+,p^-]}H_\alpha(p)&{\rm if}&p^+\le p^-
\end{array}
\right.
$$
We remark that  $g^{H_\alpha}$ are  non-decreasing in the first variable and non-increasing in the second one. Moreover, $g^{H_\alpha}(p,p)=H_\alpha(p)$ for $\alpha=R,L$. For $j\in \mathbb Z$, we define
$$p^n_{j+\frac 12}=\frac{u^n_{j+1}-u^n_{j}}{\Delta x}.$$
The numerical scheme is then given by
\begin{equation}\label{eq:scheme-HJ}
\left\{\begin{array}{ll}
\displaystyle{\frac {u^{n+1}_{j}-u^n_{j}}{\Delta t}+g^{H_L}  \left( p^n_{j-\frac 12},p^n_{j+\frac 12} \right)=0}&{\rm for}\; j\le -1,\\
\\
\displaystyle{\frac {u^{n+1}_{j}-u^n_{j}}{\Delta t}+g^{H_R}  \left( p^n_{j-\frac 12},p^n_{j+\frac 12} \right)=0}&{\rm for}\; j\ge 1,\\
\\
\displaystyle{\frac {u^{n+1}_{j}-u^n_{j}}{\Delta t}+F_0  \left(p^n_{j-\frac 12}, p^n_{j+\frac 12} \right)=0}&{\rm for}\; j=0\\
\end{array}
\right.
\end{equation}
completed with the initial condition
$$u^0_j=u_0(j\Delta x)\quad {\rm for}\; j\in \mathbb Z.$$
{For $\Delta=(\Delta t,\Delta x)$, let
$$u_{\Delta}(t,x) := \sum_{n \in \mathbb{N}} \ind_{[t_n, t_{n+1})}(t) \ind_{[x_j, x_{j+1})}(x) \left[ u_j^n + \frac{u_j^{n+1} - u_j^n}{\Delta x} (x-x_n) \right].$$}
We then have the following convergence result.
\begin{theorem}[Numerical approximation for Hamilton-Jacobi equations]\label{th:CVHJ}
Let $T>0$ and $u_0$ be Lipschitz continuous. We assume that the $H_\alpha$ satisfy \eqref{ass:H} and $F_0$ satisfies \eqref{ass:F0} or is of the form \eqref{eq:F-A}. Let $u^n_{j}$ be the solution of the scheme \eqref{eq:scheme-HJ} and $u$ be the solution of the Hamilton-Jacobi equation \eqref{eq:HJstrong} with the relaxed junction condition $F_{A_{F_0}}$.
 Let $L_{\mathcal{H}} := \max( L_{H_L}, L_{H_R}, ||\partial_{p_1} F_0||_\infty, ||\partial_{p_2} F_0||_\infty  )$ where $L_{H_\alpha}$ is the Lipschitz constant of $H_\alpha$. We also assume that the CFL condition
    \begin{equation} \label{eq:CFLSCL}
    \frac{\Delta x}{\Delta t} \geq 2  L_{\mathcal{H}}
    \end{equation}
holds. Then {$u_\Delta$} converges locally uniformly to $u$.
\end{theorem}

\begin{proof}
Recalling that by Theorem \ref{th:equivHJ}, the solution to \eqref{eq:HJstrong} with $A=A_{F_0}$ is also the solution to \eqref{eq:HJ}, the proof is a  consequence of \cite[Theorem 1.1 or Theorem 1.2]{guerand} remarking that the two schemes are identical. The main difference with the result in \cite{guerand} is that in that paper, the network is composed of two outgoing edges, but it's rather easy to come back to this setting. Indeed, if we set, for $x\ge 0$,
$$v^\alpha (t,x)=\left\{\begin{array}{lll}
u(t,-x)&{\rm if} &\alpha=L\\
u(t,x)&{\rm if}&\alpha=R
\end{array}
\right.$$
then $v^\alpha$ is solution of
\begin{equation}\label{eq:HJ2}
\left\{\begin{array}{lll}
  v^\alpha_t+\tilde H_\alpha (v_x)=0  &{\rm in}&(0,T)\times (0,+\infty), \; \alpha=R,L  \\
   v^\alpha_t+\tilde F_0(v^L_x,v^R_x)=0 & {\rm in}&(0,T)\times \{0\}
\end{array}\right.
\end{equation}
where $\tilde H_L(p)=H_L(-p)$, $\tilde H_R=H_R$, $\tilde F_0(p_1,p_2)=F_0(-p_1,p_2)$.
Setting $v^{L,n}_j=u^n_{-j}$ and $v^{R,n}_j=u^n_{j}$ for $j\ge 0$, an easy computation, using that $g^{\tilde H_L}(p_1,p_2)={g^{H_L}(-p_2,-p_1)}$, shows that $v^{\alpha,n}_j$ is solution of the following scheme
$$\left\{\begin{array}{ll}
\displaystyle{\frac {v^{\alpha,n+1}_{j}-v^{\alpha,n}_{j}}{\Delta t}+g^{\tilde H_\alpha}  \left( p^{\alpha,n}_{j-\frac 12},p^{\alpha,n}_{j+\frac 12} \right)=0}&{\rm for}\; j\ge 1,\;\quad  \alpha=L,R\\
\\
\displaystyle{\frac {v^{\alpha, n+1}_{j}-v^{\alpha,n}_{j}}{\Delta t}+\tilde F_0  \left(p^{L,n}_{j+\frac 12}, p^{R,n}_{j+\frac 12} \right)=0}&{\rm for}\; j=0,\;\quad  \alpha=L,R,\quad\mbox{with}\quad   v_0^{L,n}=v_0^{R,n}\\
\end{array}
\right.
$$
where
\begin{equation}\label{eq:DefpDelta}
   p^{\alpha,n}_{j+\frac 12}=\frac{v^{\alpha,n}_{j+1}-v^{\alpha,n}_{j}}{\Delta x}.
\end{equation}

On the contrary, the scheme proposed in \cite{guerand} to solve \eqref{eq:HJ2} writes
\begin{equation}
\left\{\begin{array}{ll}
\displaystyle{\frac {v^{\alpha,n+1}_{j}-v^{\alpha,n}_{j}}{\Delta t}+\max  \left( \tilde H^+_\alpha\left( p^{\alpha,n}_{j-\frac 12} \right),\tilde H^-_\alpha\left(p^{\alpha,n}_{j+\frac 12}  \right)\right)=0}&{\rm for}\; j\in \mathbb N,\quad \alpha=1,2, \\
\\
\displaystyle{\frac {v^{\alpha,n+1}_{j}-v^{\alpha,n}_{j}}{\Delta t}+\tilde F_0  \left(p^{L,n}_{j+\frac 12}, p^{R,n}_{j+\frac 12} \right)=0}&{\rm for}\;  j=0,\quad \alpha=1,2\; \quad \mbox{with}\quad v_0^{1,b}=v_0^{2,n}.\\
\end{array}
\right.
\end{equation}
The rest of the proof is then a direct consequence of the following lemma which proof is postponed.
\begin{lemma}[Equivalent formulation of the Godunov flux]\label{lem:scheme}
For a general convex hamiltonian $H$, we have
$$g^H(p_1,p_2)=\max(H^+(p_1),H^-(p_2)).$$
\end{lemma}
This shows that the two schemes for \eqref{eq:HJ2} are equivalent and so, using \cite[Theorem 1.1 or Theorem 1.2]{guerand}, this shows that {"$v^{\alpha,n}_i$ converges to $v^\alpha$" locally uniformly and so, by a change of variable, "$u^n_i$ converges to $u$" in the sense of Theorem \ref{th:CVHJ}.}
This ends the proof of the theorem.
\end{proof}
It remains to show the lemma.
\begin{proof}[Proof of Lemma \ref{lem:scheme}]
We denote by $p_0$ the minimum point of $H$ so that $H$ is non-increasing on {$(-\infty,p_0]$} and non-decreasing on $[p_0,+\infty)$ and we distinguish several cases:\newline
\noindent{\bf Case 1: $p_1\le p_0\le p_2$}. In that case $H^+(p_1)= H(p_0)=H^-(p_2)$ and
$$g^{H}(p_1,p_2)=\min_{[p_1,p_2]}H=H(p_0)=\max(H^+(p_1),H^-(p_2)).$$
\noindent{\bf Case 2: $p_0\le p_1\le p_2$}. In that case $H^+(p_1)=H(p_1)$,  $H^-(p_2)=H(p_0)$ and
$$g^{H}(p_1,p_2)=\min_{[p_1,p_2]}H=H(p_1)=\max(H^+(p_1),H^-(p_2)).$$
\noindent{\bf Case 3: $p_1\le p_2\le p_0$}. This case is similar to the previous one.\medskip

\noindent{\bf Case 4: $p_2\le p_0\le p_1$}. In that case $H^+(p_1)= H(p_1)$, $H^-(p_2)=H(p_2)$ and
$$\max(H(p_1),H(p_2))=\max_{[p_2,p_1]}H=g^{H}(p_1,p_2).$$
\noindent{\bf Case 5: $p_0\le p_2\le p_1$}. In that case $H^+(p_1)= H(p_1)$, $H^-(p_2)=H(p_0)$ and
$$\max(H(p_1),H(p_0))=H(p_1)=\max_{[p_2,p_1]}H=g^{H}(p_1,p_2).$$
\noindent{\bf Case 6: $p_2\le p_1\le p_0$}. This case is similar to the previous one.
\medskip
\end{proof}

\subsection{Numerical scheme for the scalar conservation law equation \eqref{eq:SCL}}\label{subsec:4.2}

Given $u_0$ satisfying \eqref{eq:u0}, we consider $\rho_0 :=  (u_0)_x$ and its discretized version
$$p^0_{j+1/2}= \frac{u_{j+1}^0 - u_j^0}{\Delta x} = \frac{u_0(x_{j+1}) - u_0(x_{j}) }{\Delta x} = \frac{1}{\Delta x}\int_{x_{j}}^{x_{j+1}} \rho_0(y) \d y.$$
We now want to describe the numerical scheme for \eqref{eq:SCL}. This scheme is directly derived from the scheme \eqref{eq:scheme-HJ}. Indeed, recalling the definition of $p^n_{j+1/2}$ in \eqref{eq:DefpDelta}, we can write
\begin{equation}\label{eq:scheme-SCL}
\left\{\begin{array}{ll}
\displaystyle{p^{n+1}_{j+\frac 12}=p^n_{j+\frac 12}-\frac{\Delta t}{\Delta x}
\left(g^{H_L}\left(p^n_{j+\frac 12},p^n_{j+\frac 32}\right)-g^{H_L}\left(p^n_{j-\frac 12}, p^n_{j+\frac 12}\right)\right)}&{\rm for}\; j\le-2\\
\\
\displaystyle{p^{n+1}_{j+\frac 12}=p^n_{j+\frac 12}-\frac{\Delta t}{\Delta x}
\left(g^{H_R}\left(p^n_{j+\frac 12},p^n_{j+\frac 32}\right)-g^{H_R}\left(p^n_{j-\frac 12}, p^n_{j+\frac 12}\right)\right)}&{\rm for}\; j\ge 1\\
\\
\displaystyle{p^{n+1}_{j+\frac 12}=p^n_{j+\frac 12}-\frac{\Delta t}{\Delta x}
\left(g^{H_R}\left(p^n_{j+\frac 12},p^n_{j+\frac 32}\right)-F_0\left(p^n_{j-\frac 12}, p^n_{j+\frac 12}\right)\right)}&{\rm for}\; j=0\\
\\
\displaystyle{p^{n+1}_{j+\frac 12}=p^n_{j+\frac 12}-\frac{\Delta t}{\Delta x}
\left(F_0\left(p^n_{j+\frac 12},p^n_{j+\frac 32}\right)-g^{H_L}\left(p^n_{j-\frac 12}, p^n_{j+\frac 12}\right)\right)}&{\rm for}\; j=- 1.\\
\end{array}
\right.
\end{equation}
For notations' sake, we also denote by $\mathcal{F}^n_j$ the right-hand side of the above scheme such that for any $n,j$, we have
$$\rho^{n+1}_{j+1/2} = \mathcal{F}^n_j(\rho^n_{j-1/2},\rho^n_{j+1/2},\rho^n_{j+3/2}).$$
We denote $\Delta = (\Delta x, \Delta t)$ and
\begin{equation}\label{eq:300}
   p_{\Delta} := \sum_{n\in \mathbb{N}} \sum_{j\in\mathbb{Z}} p_{j+1/2}^n \ind_{[t_n, t_{n+1})\times [x_j, x_{j+1})} .
\end{equation}

For this scheme we have the following convergence result
\begin{theorem}[Numerical approximation for SCL]\label{th:CVSCL}
 Let $u_0$ satisfy \eqref{eq:u0}, $H_{L,R}$ satisfy \eqref{ass:H} and $F_0$ satisfy  \eqref{ass:F0}. Suppose also that the CFL condition \eqref{eq:CFLSCL} is satisfied  and that
 \begin{equation}\label{eq:CFL-bis}
   \frac{\Delta t}{\Delta x} \frac \delta 2 M \leq 1,
   \end{equation}
where $M=\max(|a_L|,|c_L|,|a_R|, |c_R|)$ and $\delta$ is introduced in \eqref{ass:H}.
   Then $(p_\Delta)_\Delta$ converges almost everywhere  as $\Delta \longrightarrow (0,0)$ to $\rho \in L^{\infty}$, the unique solution  to \eqref{eq:SCL-strong}, in the sense of Definition \ref{def:SCLSolTrace}, with $A=A_{F_0}$ and $A_{F_0}$ given in Lemma \ref{lem:AF0}.
\end{theorem}
\begin{remark}
   This result is rather classical if we take {$F_0:=\bar F_{A}$ for $A\in [H_0,0]$} in the numerical scheme \eqref{eq:scheme-SCL} and the proof of convergence has been written in a similar setting in various sources including \cite{AGS2010}, \cite{AbrahamPreprint} and \cite{Abraham2022}. The result we present here is stronger. Indeed, we put the desired condition $F_0$ in the scheme and we show that the numerical solution converges to the solution with the relaxed junction condition {$\bar F_{A_{F_0}}$}. The strategy of the proof is similar to the classical case, but for completeness' sake we rewrite it here, putting most of the heavy computations in Appendix.
\end{remark}

We first present the different lemmas that we piece together in order to get Theorem \ref{th:CVSCL}.

\begin{lemma}\label{lem:PBorne}\emph{(Monotonicity and stability)} \\
{Let $u_0$ satisfy \eqref{eq:u0}, $H_{L,R}$ satisfy \eqref{ass:H} and $F_0$ satisfy  \eqref{ass:F0}.} Suppose also that the CFL condition \eqref{eq:CFLSCL} is satisfied. Then the numerical scheme \eqref{eq:scheme-SCL} is monotone. That is to say  $\mathcal{F}^n_j$ is non-decreasing with respect to each of its three variables. Furthermore, the scheme is stable, {namely} we have
\begin{equation}\label{eq:stability}
\forall n \in \mathbb{N}, \forall j \in\mathbb{Z}, \; p^n_{j+1/2} \in
\left\{\begin{array}{ll}
[a_L,c_L] &{\rm if}\; j\le -1\\
{[a_R,c_R]} &{\rm if}\; j\ge 0.
\end{array}\right.
\end{equation}
\end{lemma}
\begin{proof}
We begin to prove  the {monotonicity}.
Fix $n,j$.
Recall that the Godunov flux $g^H$ and the junction condition $F_0$ are non-decreasing with respect to their first argument and non-increasing with respect to their second one.  Then,
\begin{align*}
   \forall v,w \in \mathbb{R}, \; u \mapsto \mathcal{F}^n_j (u,v,w) \textrm{ is non-decreasing,} \\
   \forall u,v \in \mathbb{R}, \; w \mapsto \mathcal{F}^n_j (u,v,w) \textrm{ is non-decreasing.}
\end{align*}
Notice that, for a given $H$, the derivative of the Godunov flux $g^H$ is bounded by the Lipschitz constant $L_H$ of $H$
\begin{align*}
   \partial_{p_1} g^H(p_1,p_2) \in [0; L_H] \quad \partial_{p_2} g^H(p_1,p_2) \in [-L_H,0]
\end{align*}
Recalling that $L_{\mathcal{H}} := \max( L_{H_L}, L_{H_R}, ||\partial_{p_1} F_0(p_1,p_2)||_\infty, ||\partial_{p_2} F_0(p_1,p_2)||_\infty )$, we also have
$$\partial_{p_1}F_0(p_1,p_2)\in[0,L_{\mathcal H}]\quad \partial_{p_1}F_0(p_1,p_2)\in[-L_{\mathcal H},0].$$
Then
\begin{align*}
\partial_v \mathcal{F}^n_j (u,v,w) &\geq 1 - \frac{\Delta_t}{\Delta_x} \left( L_{\mathcal{H}} -(-L_{\mathcal{H}}) \right)\\
&\geq 1 - 2 \frac{\Delta_t}{\Delta_x} L_{\mathcal{H}}.
\end{align*}
Since the CFL condition \eqref{eq:CFLSCL} is satisfied, we recover that $v \mapsto \mathcal{F}^n_j (u,v,w)$ is non-decreasing.

\smallskip

We now prove the stability result by induction on $n$.
First, by assumption \eqref{eq:u0}, the property \eqref{eq:stability} is true for $n=0$.
Fix $n \in \mathbb{N}$ such that \eqref{eq:stability} holds true for $n$. We recall that
$$\rho^{n+1}_{j+1/2} = \mathcal{F}^n_j(\rho^n_{j-1/2},\rho^n_{j+1/2},\rho^n_{j+3/2}).$$
If $j\ge 1$, by monotonicity of the scheme, we then have
$$\rho^{n+1}_{j+1/2} \ge  \mathcal{F}^n_j(a_R, a_R,a_R)=a_R-\frac {\Delta t}{\Delta x}(H(a_R)-H(a_R))=a_R$$
and
$$\rho^{n+1}_{j+1/2} \le  \mathcal{F}^n_j(c_R, c_R,c_R)=c_R-\frac {\Delta t}{\Delta x}(H(c_R)-H(c_R))=c_R.$$
In the same way, if $j=0$, we get
$$\rho^{n+1}_{j+1/2} \ge  \mathcal{F}^n_j(a_L, a_R,a_R)=a_R-\frac {\Delta t}{\Delta x}(H(a_R)-F_0(a_L,a_R))=a_R$$
and
$$\rho^{n+1}_{j+1/2} \le  \mathcal{F}^n_j(c_L, c_R,c_R)=c_R-\frac {\Delta t}{\Delta x}(H(c_R)-F_0(c_L,c_R))=c_R,$$
where we used Assumption \eqref{ass:F0} to get that $F_0(a_L,a_R)=F_0(c_L,c_R)=0$. Using the same arguments, we get also the result for $j\le 1$. This ends the proof of the lemma.
\end{proof}

Recall that, associated to the entropy $p\mapsto |p-k|$, is the entropy flux
$$p\mapsto \mbox{sign}(p-k)\cdot \left\{H(p)-H(k)\right\}=H(p\wedge k)-H(p\vee k)$$
where we used the notation, for any $a,b \in \mathbb{R}$,  $a\vee b=\max(a,b)$ and $a\wedge b=\min(a,b)$.
This naturally suggests the following result.

\begin{lemma}[Discrete entropy inequalities]\label{lem:InegEntropiqueDiscrete}
 Let $u_0$  {satisfy \eqref{eq:u0}, $H_{L,R}$ satisfy \eqref{ass:H} and $F_0$ satisfy \eqref{ass:F0}}. Suppose also that the CFL condition \eqref{eq:CFLSCL} is satisfied. Let $T>0$ and $(p_\Delta)_\Delta$ be defined by \eqref{eq:300}.
For any $(k_L,k_R) \in Q$, writing $k_\Delta = k_L \ind_{j \leq -1} + k_R \ind_{j \geq 0}$, we set
\begin{align}\label{eq:301}
\Phi_j^n(k_\Delta)=\left\{ \begin{matrix}
   g^{H_L}(p_{j-1/2}^n \vee k_L, p_{j+1/2}^n \vee k_L) - g^{H_L}(p_{j-1/2}^n \wedge k_L, p_{j+1/2}^n \wedge k_L) & \textrm{ if } & j \leq -1 \\
   g^{H_R}(p_{j-1/2}^n \vee k_R, p_{j+1/2}^n \vee k_R) - g^{H_R}(p_{j-1/2}^n \wedge k_R, p_{j+1/2}^n \wedge k_R) & \textrm{ if } & j \geq 1 \\
   F_0(p_{j-1/2}^n \vee k_L, p_{j+1/2}^n \vee k_R) - F_0(p_{j-1/2}^n \wedge k_L, p_{j+1/2}^n \wedge k_R) & \textrm{ if } & j = 0.\\
\end{matrix}\right.
\end{align}
 We also set
\begin{align*}
   \Phi_{\Delta}(k_\Delta) := \sum_{n\in\mathbb{N}} \sum_{j\in\mathbb{Z}} \Phi_{j}^n(k_\Delta) \ind_{[t_n, t_{n+1})\times [x_j, x_{j+1})}.
\end{align*}
Then, for any $\phi \in {C}^{\infty}_c((0,T)\times\mathbb{R})$ non-negative, we have, with $p_\Delta$ defined in \eqref{eq:300},
   \begin{align}\label{eq:entro-discrete}
       \int_0^T \int_{\mathbb{R}}\left( \left| p_{\Delta} - k_\Delta \right|  \phi_t + \Phi_\Delta(k_\Delta)  \phi_x \right)\d t \d x
       + \int_{\mathbb{R}} \left| p_{\Delta}(0,x) - k_\Delta \right| \phi(0,x) \d x \notag \\
       + \int_0^T R_{F_0}(k_L,k_R) \phi(t,0) \d t \geq O (\Delta x) + O (\Delta t),
   \end{align}
   where $R_{F_0}(k_L,k_R) := |H_L(k_L) - F_0(k_L,k_R)| + |H_R(k_R) - F_0(k_L,k_R)|$.
\end{lemma}

\begin{remark}
The proof of this lemma is pretty straightforward and derives directly from the monotonicity proven in Lemma \ref{lem:PBorne}. Since it contains long computations, we postponed it to the Apppendix.
\end{remark}

Finally, in order to get the desired convergence, we also need the compactness of $(p_{\Delta})_{\Delta}$. We use the following lemma, which proof is also postponed to the Appendix.

\begin{lemma}[Compactness of $\rho_\Delta$]\label{lem:Compactness}
{Let $u_0$ satisfy \eqref{eq:u0}, $H_{L,R}$ satisfy \eqref{ass:H} and $F_0$ satisfy \eqref{ass:F0}.}
     For any $l$, let $(\Delta_l)_l$ verify the CFL condition \eqref{eq:CFLSCL}  {and \eqref{eq:CFL-bis}.}
   Then, there exists $\rho \in L^{\infty}$ and a subsequence also denoted $(p_{\Delta_l})_l$ such that
   $$p_{\Delta_l} \longrightarrow \rho \quad {\textrm{ a.e. as $\Delta_l\to 0$.}}$$
\end{lemma}

We are now in a position to prove Theorem \ref{th:CVSCL}.
\begin{proof}[Proof of Theorem \ref{th:CVSCL}]
   First, using Lemma \ref{lem:Compactness}, we take a subsequence of $p_\Delta$ that converges to $\rho \in L^\infty$ a.e.. We now want to prove that $\rho$ is a solution to \eqref{eq:SCL-strong} in the sense of Definition \ref{def:SCLSolTrace}. The first point of Definition \ref{def:SCLSolTrace} is classical and we skip it.\medskip

 Let $\phi\in C^\infty_c([0,T)\times \mathbb R^-)$ be non-negative and  $(k_L,k_R) \in Q$. We first want to prove that
   \begin{equation}\label{eq:500}
   \int_0^T \int_{\mathbb{R}^-} \Phi_\Delta (k_\Delta) \phi_x \d t \d x \longrightarrow \int_0^T \int_{\mathbb{R}^-} \sign(\rho - k_L) \left[ H_L(\rho) - H_L(k_L) \right]  \phi_x \d t \d x.
   \end{equation}
   Let $x < 0$ and $t \in [0,T)$ such that $p_\Delta (t,x) \longrightarrow \rho$. Then, for any $\Delta x$, there exists $j \leq -1$ such that $p_\Delta (t,x) = p_{j+1/2}^n$ and

   \begin{align*}
       \Phi_\Delta(k_\Delta)(t,x) =& g^{H_L}(p_\Delta(t,x - \Delta x) \vee k_L, p_\Delta (t,x) \vee k_L) - g^{H_L}(p_\Delta(t,x - \Delta x) \wedge k_L, p_\Delta (t,x) \wedge k_L) \\
       =& g^{H_L}(p_\Delta(t,x - \Delta x) \vee k_L, p_\Delta (t,x) \vee k_L) - g^{H_L}(p_\Delta (t,x) \vee k_L, p_\Delta (t,x) \vee k_L) \\
       &+ \sign(p_\Delta(t,x) -k_L) \left[ H_L(p_\Delta(t,x)) -H_L(k_L) \right] \\
       &+ g^{H_L}(p_\Delta (t,x) \wedge k_L, p_\Delta (t,x) \wedge k_L) -g^{H_L}(p_\Delta(t,x - \Delta x) \wedge k_L, p_\Delta (t,x) \wedge k_L).
   \end{align*}
   Using the Lipschitz bound on the Godunov flux, we recover:
   $$\int_0^T \int_\mathbb{R^-} \Phi_\Delta(k_\Delta)  \phi_x \d t \d x = \int_0^T \int_{\mathbb{R}^-} \sign(p_\Delta - k_L) \left[ H_L(p_\Delta) - H_L(k_L) \right]  \phi_x \d t \d x + \mathcal I_1$$
   where
   \begin{align*}
       |\mathcal I_1| &\leq 2 L_\mathcal{H} \int_0^T \int_{\mathbb{R}^-} \left| p_\Delta (t,x) - p_\Delta (t,x -\Delta x)\right| | \phi_x(t,x)| \d t \d x \\
       &\leq 2 L_\mathcal{H} \int_0^T \int_{\mathbb{R}^-} \left| p_\Delta (t,x) \phi_x(t,x) - p_\Delta (t,x -\Delta x) \phi_x(t,x)\right|  \d t \d x \\
       &\leq 2 L_\mathcal{H} \bigg[ \int_0^T \int_{\mathbb{R}^-} \left|  p_\Delta (t,x) \phi_x(t,x) - p_\Delta (t,x -\Delta x)  \phi_x(t,x - \Delta x)\right|  \d t \d x \\
       &+ \int_0^T \int_{\mathbb{R}^-} \left|  p_\Delta (t,x -\Delta x) \phi_x(t,x) - p_\Delta (t,x -\Delta x)  \phi_x(t,x - \Delta x)\right|  \d t \d x
       \bigg] \\
       &=: \mathcal I_2+\mathcal I_3.
   \end{align*}
   First, notice that
   $$\mathcal I_3 \leq 2 L_\mathcal{H} \int_0^T \int_{\mathbb{R}^-} \left| p_\Delta (t,x -\Delta x)\int_{x - \Delta x}^x  \phi_{xx}(t,y) \d y \right|  \d t \d x \leq 2 L_\mathcal{H} ||p_\Delta||_{L^\infty} || \phi_{xx} ||_{L^1} \Delta x.$$
   Now, since $p_\Delta \longrightarrow \rho$ a.e. and $|p_\Delta  \phi_x| \leq C | \phi_x| \in L^1((0,T)\times \mathbb{R})$, the sequence $(p_\Delta  \phi_x)_\Delta$ is convergent {to $\rho \phi_x$ in $L^1((0,T)\times \R)$.
From Frechet-Kolmogorov Theorem, we recall that
$$\lim_{\Delta x \rightarrow 0} ||\tau_{\Delta x} (\rho  \phi_x) - \rho  \phi_x||_{L^1((0,T)\times\mathbb{R})} = 0,$$
   where $\tau_{\Delta x}f(x)=f(x-\Delta x)$. It is then easy to see that} $\mathcal I_2 = o(1)$ when $\Delta \longrightarrow (0,0)$. This implies \eqref{eq:500}.

Then, for  $\phi \in {C}^\infty_c([0,T)\times\mathbb{R}^-)$, passing to the limit in \eqref{eq:entro-discrete}, we get
$$
\iint_{(0,T)\times \mathbb{R}^{-}} |\rho-k_L| \phi_t + \sign(\rho-k_L)\left[H_{L}(\rho) - H_{L}(k_L)\right] \phi_x + \int_{\mathbb{R}^-} |\rho_0(x) -k_L | \phi(0,x) \d x \geq 0.
$$
The analogous result holds if $\phi$ is compactly supported in $[0,T) \times (0,+\infty)$. Note that, when treating this case, we need to consider a $\Delta_0$ such that for any $\Delta x \leq \Delta_0$, $p_\Delta (t,x) = p_{j+1/2}^n$ with $j\geq 1$. We can however choose $\Delta x$ to be small enough such that $\phi = 0$ on $(0,\Delta x)$ and recover the analogous inequalities. So the second condition in Definition \ref{def:SCLSolTrace} is satisfied. \medskip

We now want to prove the third point. Let ${(k_L,k_R)}\in \mathcal E_{A_{F_0}}$, where $\mathcal E_{A_{F_0}}$ is defined in \eqref{eq:EA}. In particular, we have $R_{F_0}(k_L,  k_R) =0$.
Using  the same reasoning as before with  $\phi \in {C}^\infty_c((0,T) \times \mathbb{R})$, we get, using the notation
$ \mathcal{H}(x,p)  := H_L(p) \cdot \ind_{\mathbb{R}^-}(x) +H_R(p)\cdot  \ind_{\mathbb{R}^+}(x)$ and $k(x)  :=k_L \cdot  \ind_{\mathbb{R}^-}(x)+  k_R\cdot \ind_{\mathbb{R}^+}(x)$, that

$$\int_0^T \int_\mathbb{R} \Phi_\Delta (k_\Delta) \phi_x \d t \d x \longrightarrow \int_0^T \int_{\mathbb{R}\backslash\{0\}} \sign(\rho -  {k}) \left[ \mathcal{H}{(x,\rho)} - \mathcal{H}(x,k) \right]  \phi_x \d t \d x $$
and
$$
\iint_{(0,T)\times \mathbb{R}} |\rho- k| \phi_t + \sign(\rho -  k) \left[ \mathcal{H}(x,\rho) - \mathcal{H}(x,k) \right] \phi_x \geq 0
$$
Using a sequence of test functions focusing on $x=0$, we then recover
$$
   \int_0^T \left[q_L(\gamma_L \rho,  k_L) - q_R(\gamma_R \rho,  k_R) \right] \phi(t,0) \d t \geq 0,
$$   where the $q_\alpha$ are defined  in Lemma \ref{lem:generate}.
Then, for almost every $t$,
$$q_L(\gamma_L \rho,  k_L) - q_R(\gamma_R \rho,  k_R) \geq 0.$$
Using Lemma \ref{lem:generate} {and the fact that $(\gamma_L \rho,\gamma_R \rho)\in Q$ a.e. on $(0,T)$}, we deduce that {$(\gamma_L \rho, \gamma_R \rho) \in \mathcal{G}_{A_{F_0}}$ a.e. on $(0,T)$} and we recover that $\rho$ satisfies the third condition of Definition \ref{def:SCLSolTrace}.
Finally the uniqueness of $\rho$ follows from the first point of Theorem \ref{th:eqSolEntropy}.

\smallskip

\end{proof}

{We now state and prove the following result.
\begin{lemma}[Completeness of $\mathcal G_A$]\label{lem::t4}
Under the assumptions of Proposition \ref{pro:propertiesGAF0}, the $L^1$-dissipative germ $\mathcal G_A$ is complete.
\end{lemma}
\begin{proof}[Proof of Lemma \ref{lem::t4}]
Consider any $k=(k_L,k_R)\in Q$. In order to show the completeness of $\mathcal G_A$, we simply have to show the existence of a $\mathcal G_A$-entropy solution to (\ref{eq:SCL-strong}) with initial data $\rho_0=k_L1_{(-\infty,0)}+k_R1_{[0,+\infty)}$.
The existence of such a solution follows from the construction of the function $\rho$ in the proof of Theorem \ref{th:CVSCL}. This insures that $\mathcal G_A$ is complete and ends the proof.
\end{proof}}

\section{Proof of Theorem \ref{th:main} and Theorem \ref{th:main2} using numerical schemes}\label{sec:5}
We are now able to give the proof of Theorem \ref{th:main2}.
\begin{proof}[Proof of Theorem \ref{th:main2}]
Fix $\Delta := (\Delta t, \Delta x)$ satisfying the CFL condition \eqref{eq:CFLSCL} and \eqref{eq:CFL-bis}.
Denote by $(u^n_j)_{n \in \mathbb{N}, j \in \mathbb{Z}}$ the solution of the scheme \eqref{eq:scheme-HJ}. {Recall that}
$$u_{\Delta}(t,x) := \sum_{n \in \mathbb{N}} \ind_{[t_n, t_{n+1})}(t) \ind_{[x_j, x_{j+1})}(x) \left[ u_j^n + \frac{u_j^{n+1} - u_j^n}{\Delta x} (x-x_n) \right].$$
Then, by construction (see \eqref{eq:DefpDelta}), for any $\Delta$,
$$ (u_\Delta)_x = p_\Delta$$
where $p_\Delta$ is the solution of the scheme \eqref{eq:scheme-SCL} with $\rho_0$ as initial datum.
Let $\phi \in {C}^1_c([0, + \infty) \times \mathbb{R})$.
Then we have
$$\iint u_\Delta \phi_x \d t \d x = -\iint p_\Delta \phi \d t \d x.$$
Using Theorem \ref{th:CVHJ}, we know that the scheme \eqref{eq:scheme-HJ} with $u_0$ as initial datum converges locally uniformly to $u$ the unique weak viscosity solution to \eqref{eq:HJ}.
Furthermore, $\phi_x \in {C}^0_c([0, + \infty) \times \mathbb{R})$ so we can pass to the limit in the left-hand side as $\Delta \longrightarrow (0,0)$ satisfying the CFL condition to get
$$\iint u_\Delta \phi_x \d t \d x \longrightarrow \iint u \phi_x \d t \d x.$$
On the other hand, using Theorem \ref{th:CVSCL}, we get that $p_\Delta$ converges a.e. to $\rho$ the unique solution of \eqref{eq:SCL} in the sense of Definition \ref{def:SCLSolTrace}. Also, thanks to Lemma \ref{lem:PBorne}, we know that $(p_\Delta)_\Delta$ is uniformly bounded. By dominated convergence, we also pass to the limit in the right-hand side and get that, for any test function $\phi \in {C}^1_c([0, + \infty) \times \mathbb{R})$,
$$\iint u \phi_x \d t \d x = -\iint \rho \phi \d t \d x.$$
This gives the desired result.
\end{proof}
The proof of Theorem \ref{th:main} can be obtained exactly in the same way.\\

\section{An alternative proof of Theorem \ref{th:main} using semi-algebraic functions}\label{sec:6}

Let $u$ be the viscosity solution of \eqref{eq:HJstrong} {for $A\in [H_0,0]$} and  $\rho$ be defined by
\begin{equation}\label{def.rho}
\rho(t,x):=  u_x(t,x).
\end{equation}
We would like to give a more direct proof of Theorem \ref{th:main} and show that $\rho$ is an entropy solution of \eqref{eq:SCL-strong}.
{It is easy to check that $\rho$ is already an entropy solution outside $\left\{x=0\right\}$ (see for instance \cite{CP20, KR02}).
We then focus on the junction condition at $x=0$.}
In all this section, we assume that $H_{L,R}$ satisfy \eqref{ass:H}. {We denote by $\rho_L$ and $\rho_R$ the strong traces of $\rho$ at $0$ (see \cite{panov-trace} and (\ref{eq::g8})): $\rho_L:=\gamma_L\rho$ and $\rho_R:=\gamma_R\rho$.}

We first note that formally
\begin{equation}\label{un}
H_L(u_x(t,0^-))=H_R(u_x(t,0^+))\qquad \forall  t>0.
\end{equation}
Equality \eqref{un} can be rewritten {rigorously} as
\begin{lemma}[The Rankine-Hugoniot condition]\label{lem.unbis}
We have
\begin{equation}\label{unbis}
H_L(\rho_L(t))= H_R(\rho_R(t)) \qquad  \text{a.e.}\ t>0.
\end{equation}
This common value is equal to $ -u_t(t,0)$.
\end{lemma}
Equality \eqref{unbis} makes sense since {$\rho_\alpha(t,\cdot)$ (and then also $H_\alpha(\rho(t,\cdot)$) have  strong traces} at $x=0$.
Note also that equality \eqref{unbis} is nothing else the  Rankine-Hugoniot condition at $x=0$.

\begin{proof}[Proof of Lemma \eqref{lem.unbis}] For any $\xi\in C^\infty_c((0,+\infty))$ and $h>0$ small, we have, after integrating the equation of $u$ which is satisfied a.e. (since $u$ is Lipschitz continuous)
$$
h^{-1} \int\int_{{(t,x)\in (0,+\infty)\times (0,h)}} \xi(t) H_R( u_x(t,x)) \ dx dt = h^{-1} \int_0^\infty \int_0^h \xi'(t) u(t,x) \ dx dt.
$$
By continuity of $u$, the right-hand side converges, as $h\to 0^+$, to $\int_0^\infty \xi'(t) u(t,0)dt$. The left-hand side can be rewritten as
$$
h^{-1} \int_0^\infty \int_0^h \xi(t) H_R( \rho(t,x))\ dx dt
$$
and converges to $\int_0^\infty  \xi(t) H_R( \rho_R(t))\ dt$ as $h\to 0^+$ (where {$\rho_R(t)$ is the strong trace of $\rho$ at $0^+$}).
This implies that
$$\int_0^\infty  \xi(t) H_R( \rho_R(t))\ dt=\int_0^\infty \xi'(t) u(t,0)dt.$$
In the same way, we have
$$\int_0^\infty  \xi(t) H_L( \rho_L(t))\ dt=\int_0^\infty \xi'(t) u(t,0)dt.$$
This shows \eqref{unbis}. Note in addition that, as $u$ is Lipschitz continuous,
$$
\int_0^\infty \xi(t)  u_t(t,0)dt = - \int_0^\infty \xi'(t) u(t,0)dt  = -\int_0^\infty  \xi(t) H_\alpha( \rho_\alpha(t))\ dt
$$
for $\alpha=L,R$, which proves that the common value in \eqref{unbis} is equal to $ - u_t(t,0)$.
\end{proof}

We continue by showing that the {traces of $\rho$ satisfy the first line in the second equivalent definition of  $\mathcal G_{A}$ in \eqref{eq:406}.}
\begin{lemma}[$\rho$ satisfies the first property of the germ $\mathcal G_{A}$]\label{lem:2} Assume that $u$ is a solution to \eqref{eq:HJstrong}. Then $\rho$ defined by \eqref{def.rho} satisfies
{$$
H_L(\rho_L(t))\ge  A \qquad \text{a.e.} \; t>0.
$$}
\end{lemma}

{
\begin{proof} We know by \cite[Theorem 2.11]{imbert-monneau} that $w(t):=u(t,0)$ is a viscosity subsolution of $w_t+ A\leq 0$. Thus it satisfies $w(t+\tau)-w(t)\leq -A \tau$ for any $t,\tau> 0$. Let us integrate the equation satisfied by $u$ against the test function $(s,y)\to (\tau h)^{-1}{\bf 1}_{[t,t+\tau]\times [-h,0]}(s,y)$ for $\tau,h>0$. We have, by Lipschitz continuity of $u$,
\begin{align*}
0 & = (\tau h)^{-1}\int_t^{t+\tau}\int_{-h}^0(u_t(s,y) +H_L(u_x(s,y)))dyds \\
\qquad & = (\tau h)^{-1}\int_{-h}^0  (u(t+\tau,y)-u(t,y))dy +(\tau h)^{-1}\int_t^{t+\tau}\int_{-h}^0 H_L(\rho(s,y))dyds \\
\qquad & \leq   \tau^{-1}(u(t+\tau,0) -u(t,0))+C\frac{h}{\tau}  +(\tau h)^{-1}\int_t^{t+\tau}\int_{-h}^0H_L(\rho(s,y))dyds \\
\qquad & \leq -A+C\frac{h}{\tau}  +(\tau h)^{-1}\int_t^{t+\tau}\int_{-h}^0H_L(\rho(s,y))dyds.
\end{align*}
We let $h\to 0^+$ and obtain
$$
\int_t^{t+\tau}H_L(\rho_L(s))ds \ge  A \tau,
$$
which gives the claim.
\end{proof}}

\begin{lemma}[The traces are in the germ]\label{lem.rhoinGcond} Assume that the  Lipschitz continuous viscosity solution $u$ to \eqref{eq:HJstrong} satisfies
\begin{equation}\label{hyphyp}
\text{for a.e. $t\in {(0,T)}$, $u(t,\cdot)$ has a left derivative $u_x(t,0^-)$ and a right derivative $u_x(t,0^+)$ at $0$.}
\end{equation}
Then
\begin{equation}\label{eq::g9}
(\rho_L(t), \rho_R(t))=( u_x(t, 0^-), u_x(t, 0^+))\in \mathcal G_{A}\quad \mbox{for a.e.}\quad t\geq 0.
\end{equation}
\end{lemma}

\begin{remark} The forthcoming paper \cite{Monneau2023} shows that \eqref{hyphyp} actually holds in a very general set-up (and in particular under our standing conditions). Below we prove it for semi-algebraic data only by using a representation formula.
\end{remark}

\begin{proof}
{\noindent {\bf Step 1: proof of equality in (\ref{eq::g9})}}\\
Using the definition of {strong} traces, we have
\begin{equation}\label{vasseur}
\text{ess-}\lim_{x\to 0^+} \int_0^T |\rho_L(t)-\rho(t,-x)|+ |\rho_R(t)-\rho(t,x)|\ dt =0.
\end{equation}
This implies that
$$
\text{ess-}\lim_{x\to 0^+} \int_0^T |\rho_L(t)- u_x(t,-x)|+|\rho_R(t)-u_x(t,x)|\ dt =0.
$$
Therefore, for any $\ep>0$ there exists $x_\ep>0$ such that
$$
\int_0^T\left| \rho_L(t)- u_x(t,-x)\right|+ \left|\rho_R(t)-u_x(t,x) \right|\ dt \leq \ep \quad \mbox{for a.e.}\quad x\in (0,x_\ep).
$$
Thus, after integration in space, we get
$$
\int_0^T |\rho_L(t)x+u(t,-x)-u(t,0)|+|\rho_R(t)x-u(t,x)+u(t,0)|dt \leq \ep x\quad \mbox{for all}\quad x\in (0,x_\ep).
$$
Using that $u$ is Lipschitz continuous, assumption (\ref{hyphyp}) and Lebesgue Theorem, we get therefore
$$
\int_0^T \left| \rho_L(t)- \lim_{x\to 0^+} \frac{u(t,-x)- u(t,0)}{-x}\right|+  \int_0^T\left|\rho_R(t)- \lim_{x\to 0^+} \frac{u(t,x)-u(t,0)}{x} \right|\ dt =0.
$$
This means that, for a.e. $t$,
\begin{equation}\label{eq:pente}
  {u_x(t,0^-)=}  \lim_{x\to 0^+} \frac{u(t,-x)- u(t,0)}{-x}=\rho_L(t)\qquad \text{and} \qquad  {u_x(t,0^+)=}\lim_{x\to 0^+} \frac{u(t,x)-u(t,0)}{x}= \rho_R(t).
\end{equation}
{\noindent {\bf Step 2: proof of the inclusion in (\ref{eq::g9})}\\
We already know that $-u_t(t,0)= H_L(\rho_L(t))= H_R(\rho_R(t))\ge A$ for a.e. time $t$ (see Lemma \ref{lem.unbis} and Lemma \ref{lem:2}).
Let us fix such a time $t>0$.
Our aim is to check that $(\rho_L(t),\rho_R(t))  \in \mathcal G_{A}$. We argue by contradiction, assuming that
$$H_L(\rho_L(t))> A,\quad H_L^+(\rho_L(t))<H_L(\rho_L(t))\quad  {\rm and}\quad H_R^-(\rho_R(t))<H_R(\rho_R(t)).$$
Let us fix $\ep>0$ so small that $\lambda:=H_L(\rho_L(t))-\ep>A$. We then choose $k_L^\ep$ as the smallest  solution to $H_L(k_L^\ep):=\lambda$ and $k_R^\ep$ as the largest solution to $H_R(k_R^\ep):= \lambda$. As $H_L$ and $H_R$ are convex and $H_L^+(\rho_L(t))<H_L(\rho_L(t))$ and $H_R^-(\rho_R(t))<H_R(\rho_R(t))$, we have $k_L^\ep>\rho_L(t)$ and $k_R^\ep<\rho_R(t)$. Moreover, $H_L^+(k_L^\ep)= \min H_L$, while $H_R^-(k_R^\ep)= \min H_R$. Let us define the map $w:\R\times \R\to \R$ by
$$
w(s, x)= u(t,0) + \left\{ \begin{array}{ll}
k_L^\ep x- \lambda s& \text{if}\; x\leq 0\\
k_R^\ep x -\lambda s& \text{if}\; x\geq 0
\end{array}\right.
$$
Then $w$ is a test function which is a subsolution of the Hamilton-Jacobi equation \eqref{eq:HJstrong} because,
$$
H_L(k_L^\ep)= H_R(k_R^\ep) =\lambda = - w_s$$
and using $A\in [H_0,0]$, we get
$$\max \{A, H_L^+(k_L^\ep), H_R^-(k_R^\ep)\}=  \max \{A, \min H_L, \min H_R\} = A\le \lambda=-w_s.
$$
Moreover, by \eqref{eq:pente} and the fact that $k_L^\ep>\rho_L(t)$ and $k_R^\ep<\rho_R(t)$, we get that $u(t,x)\geq w(0,x)$ if $|x|$ is small enough. Thus, by finite speed of propagation and comparison, we have $u(t+h, 0)\geq w(h,0)$ for $h>0$ small enough.
Therefore
$$
-H_L(\rho_L(t))=  u_t(t,0) \geq  w_s(0,0) =-\lambda =-H_L(\rho_L(t))+\ep,
$$
which contradicts our assumption. This proves that $(\rho_L(t),\rho_R(t))\in \mathcal G_{A}$.}
\end{proof}

We are now ready to give {an} alternative proof of Theorem \ref{th:main}.
This proof relies on semi-algebraic functions. For the reader's convenience, we recall below some {useful} facts about semi-algebraic sets and functions and we refer to \cite{coste} for a complete reference (see also \cite{coste2}).
\begin{remark}
We recall that a basic semi-algebraic set is a set defined by {a finite number of} polynomial equalities and polynomial inequalities, and a semi-algebraic set is a finite union of basic semi-algebraic sets.
The class $\mathcal{SA}_n$ of semi-algebraic subsets of $\R^n$ has the following properties:
\begin{itemize}
\item All algebraic subsets of $\R^n$ (i.e., zeros of a finite number of polynomial equalities) are in $\mathcal{SA}_n$.

\item $\mathcal{SA}_n$ is stable by finite intersection, finite union and taking complement.

\item The cartesian products of semi-algebraic sets are semi-algebraic.

\item  The Tarski-Seidenberg Theorem says that the image by the canonical projection $p:\R^{n+1}\to \R^n$ of a semi-algebraic set of $\R^{n+1}$ is a semi-algebraic set of $\R^n$.

\item By \cite[Proposition 1.12]{coste},   the closure and the interior of a semi-algebraic subset of $\R^n$ are semi-algebraic.

\item By definition, a semi-algebraic map is a map defined on a semi-algebraic set {and whose graph is} a semi-algebraic set.

\item  An important property of semi-algebraic functions is given in \cite[Theorem 2.1]{coste} (Monotonicity Theorem): If $f : (a, b) \to \R$ is semi-algebraic, then there exists a finite subdivision $a = a_0 < a_1 < \dots < a_k = b$ such
that, on each interval $(a_i, a_{i+1})$, $f$ is continuous and either constant or strictly monotone.

\item An {important} consequence of the monotonicity Theorem is given in \cite[Lemma 6.1]{coste}: {left and right derivatives of a continuous semi-algebraic} map on an open interval exist ({with values} in $\R\cup \{\pm \infty\}$).
\end{itemize}
\end{remark}

\begin{proof}[Sketch of proof of Theorem \ref{th:main}] { As Theorem \ref{th:main} has already been established by using numerical schemes, we only sketch the proof. Recall that $u$ is a viscosity solution of \eqref{eq:HJstrong}.} We have  to prove that $\rho:=  u_x$ is an {entropy} solution to \eqref{eq:SCL-strong}.
{Following  for instance \cite{CP20, KR02}, we know that} $\rho$ solves the equation in $\{x\neq 0\}$.
{It remains to check the junction condition at $x=0$.
From Lemma \ref{lem.rhoinGcond}, we just need to show that the left and right derivatives $u_x(t,0^-)$ and $u_x(t,0^+)$ are well defined for a.e. time $t$. To do so we will use a representation formula. Using this representation formula, we show the existence of these derivatives when the initial datum and hamiltonians are semi-algebraic, then conclude by an approximation argument. \\

\noindent {\bf Step 1: representation formula of the solution}\\
In order to use a representation formula, we reverse the time direction of trajectories, and for this reason, we set $\hat u(t,x)= u(T-t,x)$ and}
$$
L_\alpha(q)= \sup_{p\in \R} \left(- qp-H_\alpha(p)\right)
$$
where (with the same notation), we denote by $H_\alpha:\R\to \R$ a $C^1$, strictly convex and superlinear extension of $H_\alpha$ from the interval $[a_\alpha,c_\alpha]$ to the whole line $\R$,  for $\alpha=L,R$. This implies that $L_\alpha:\R\to \R$ is also $C^1$, strictly convex and superlinear. Let us now define
$$L(x,q):=\left\{\begin{array}{lll}
L_L(q)&{\rm if}&x <0\\
-A&{\rm if}&x=0\\
L_R(q)&{\rm if}&x >0
\end{array}
\right.
$$
Following \cite[Proposition 6.3]{imbert-monneau}, {for $t_0\le T$,} we have
$$
\hat u(t_0,x_0) =\inf_{\gamma(t_0)=x_0} \int_{t_0}^T L(\gamma(t),\dot \gamma(t)) \; dt + u_0(\gamma(T)),
$$
where the infimum is taken over the trajectories $\gamma\in H^1([t_0,T],\R)$.\medskip

If $\hat \gamma$ is optimal for $x_0$, then $\hat \gamma$ is a straight-line on each interval where it does not vanish (by optimality conditions {using $L_\alpha$ strictly convex}). As a consequence, the minimization problem boils down to minimize {for $t_0<T$ and if, for instance $x_0<0$}:
\begin{align}
{\hat u}(t_0,x_0) = \min & \Bigl\{ \min_{y\leq 0} (T-t_0)L_L\left(\frac{y-x_0}{T-t_0}\right)+u_0(y),\notag \\
& \min_{t_0< \tau_1\leq \tau_2{<} T, y\geq 0} (\tau_1-t_0) L_L\left(\frac{0-x_0}{\tau_1-t_0}\right) -{A}(\tau_2-\tau_1) + (T-\tau_2)L_R\left(\frac{y-0}{T-\tau_2}\right)+ u_0(y), \notag\\
& \min_{t_0< \tau_1\leq \tau_2{<} T, y\leq 0} (\tau_1-t_0) L_L\left(\frac{0-x_0}{\tau_1-t_0}\right) - {A}(\tau_2-\tau_1) + (T-\tau_2)L_L\left(\frac{y-0}{T-\tau_2}\right)+ {u_0(y)}\Bigr\}\notag \\
= \min & \{f_1(x_0), f_2(x_0), f_3(x_0)\}, \label{repu}
\end{align}
where $f_1$ corresponds to trajectories ending at $y\leq 0$ while $f_2$ (resp. $f_3$) corresponds to trajectories ending at $y\geq 0$ (resp. $y\leq 0$) and remaining in $x=0$ during the time interval $[\tau_1,\tau_2]$.
{Notice that (\ref{repu}) is still true for $x_0=0$, with each minimum replaced by an infimum.}
\medskip

\noindent {\bf Step 2: argument for  semi-algebraic data}\\
Here we assume that the data ($L_R$, $L_L$ and $u_0$) are semi-algebraic. We claim that the map $\hat u(t_0,\cdot)$ given by \eqref{repu} is also semi-algebraic. Let us mention that in the case of analytic data, Trélat proved in \cite{Trelat1,Trelat2} that the solution to the Hamilton-Jacobi equation is subanalytic.
%
%
%
To prove our claim,  let us show for instance that $f_2$ is semi-algebraic. Let us define the semi-algebraic set $A_2$ by
\begin{align*}
A_2:= & \Bigl\{ (t_0,x_0, \tau_1,\tau_2,y,z,u,v)\in\R^8, \; 0< t_0{<} \tau_1\leq \tau_2< T, \; (\tau_1-t_0)u=x_0<0, \;  (T-\tau_2)v=y,\\
&\qquad
y\geq 0, \;
z\geq  (\tau_1-t_0) L_L(-u) -{ A}(\tau_2-\tau_1) + (T-\tau_2) L_R(v)+u_0(y)\Bigr\}.
\end{align*}
Let $C_2$ denotes the projection of $A_2$ onto the components $(t_0,x_0,z)$. Then, by the Tarski-Seidenberg Theorem, $C_2$ is a semi-algebraic set. Note that $C_2$ is also, by definition, the epigraph of $f_2$. Therefore the subgraph of $f_2$ (which is the closure of the complement of $C_2$) and its graph (intersection of the  epigraph and subgraph) are also semi-algebraic. Thus $f_2$ is a semi-algebraic map. By stability of semi-algebraic sets by finite union, we deduce that $\hat u(t_0,\cdot)$ is semi-algebraic on $(-\infty, 0)$. Moreover the function $\hat u(t_0,\cdot)$ is continuous at $x_0=0$. Hence $\hat u(t_0,\cdot)$
is also semi-algebraic on $(-\infty,0]$.
A similar argument shows that it is also semi-algebraic on $[0,\infty)$. Because the union of semi-algebraic sets is semi-algebraic, we deduce that $\hat u(t_0,\cdot)$ is semi-algebraic on $\R$. This implies that $u(t, \cdot)$ is semi-algebraic on $\R$ for any $t\in (0,T)$.

Using \cite[Lemma 6.1]{coste}, we then deduce that the limits
$$
u_x(t,0^-):=\lim_{h\to 0^-} \frac{u(t,h)-u(t,0)}{h}\quad \mbox{and}\quad u_x(t,0^+):=\lim_{h\to 0+} \frac{u(t,h)-u(t,0)}{h}
$$
exist at any time $t\in (0,T)$.
Therefore \eqref{hyphyp} holds. We can then conclude by Lemma \ref{lem.rhoinGcond} that $(\rho_L(t),\rho_R(t))\in \mathcal G_{A}$ for a.e. $t\in [0,T]$.\medskip

\noindent {\bf Step 3: argument in the general case}\\
One can check\footnote{This is the point where the proof is sketchy: the actual construction of $H^\varepsilon_L$, $H^\varepsilon_R$,  and $u_0^\varepsilon$ requires some work and has to be done with care.} that it is possible to approximate our data $H_L$, $H_R$, $u_0$ by semi-algebraic data $H^\varepsilon_L$, $H^\varepsilon_L$,  and $u_0^\varepsilon$ satisfying our standing assumptions (with locally uniform convexity for $H^\varepsilon_L$ and $H^\varepsilon_L$). By the previous step, we know that, if $u^\varepsilon$ is the solution to the HJ equation associated with these perturbed data, then $\rho^\varepsilon= u^\varepsilon_x$ solves the associated SCL. To conclude, we only need to pass to the limit: indeed, $u^\varepsilon$ converges locally uniformly to the solution $u$ of the HJ equation \eqref{eq:HJstrong}, while $\rho^\varepsilon$ converges in $L^1_{loc}$ to the entropy solution $\rho$ of \eqref{eq:SCL-strong}. We infer therefore that $u_x$, which is the weak limit of $u^\ep_x$, is equal to the solution $\rho$ of  \eqref{eq:SCL-strong}.
\end{proof}

\section{Appendix}\label{s6}
\subsection{Proof of the discrete entropy inequalities for the SCL numerical scheme}\label{s6a}
Before proving that the scheme satisfies the discrete entropy inequalities stated in Lemma \ref{lem:InegEntropiqueDiscrete}, we prove the following discrete entropy inequalities, independent of the test function.
\begin{lemma}[First discrete entropy inequalities]\label{lem:A1}
The numerical scheme \eqref{eq:scheme-SCL} satisfies the following discrete entropy inequalities: for all $n \in \mathbb{N}$, $j \in \mathbb{N}$ and $(k_L, k_R)\in Q$, set $k_\Delta = k_L \ind_{j \leq -1} + k_R \ind_{j \geq 0}$. Then
\begin{align*}
   &\frac{|p_{j+1/2}^{n+1} - k_\Delta | - |p_{j+1/2}^{n} - k_\Delta |}{\Delta t} + \frac{\Phi_{j+1}^n(k_\Delta) - \Phi_{j}^n(k_\Delta)}{\Delta x}
   &\leq \left\{ \begin{array}{ll}
  \displaystyle \frac{R_L}{\Delta x} & \textrm{ if }  j = -1\\
  \\
  \displaystyle \frac{R_R}{\Delta x} & \textrm{ if }  j = 0\\
  \\
   0 & \textrm{ otherwise }\\
   \end{array} \right.
\end{align*}
where
$$
   R_{\alpha} = \left|H_{\alpha}(k_\alpha) - F_0(k_L,k_R) \right|, \quad \alpha=L,R,
$$
and $\Phi^n_j(k_\Delta)$ is defined in \eqref{eq:301}.
\end{lemma}
\begin{proof}
Let $k \in \mathbb{R}$. Fix $n\in\mathbb{N}$, $j\in\mathbb{Z}$ such that $j\neq0,-1$. We have, using the monotonicity of the scheme,
\begin{align*}
|p_{j+1/2}^{n+1} - k | &= p_{j+1/2}^{n+1}\vee k - p_{j+1/2}^{n+1}\wedge k \\
&= \mathcal{F}_j^n( p_{j-1/2}^{n},p_{j+1/2}^{n},p_{j+3/2}^{n} ) \vee \mathcal{F}_j^n(k,k,k) - \mathcal{F}_j^n( p_{j-1/2}^{n},p_{j+1/2}^{n},p_{j+3/2}^{n} ) \wedge \mathcal{F}_j^n(k,k,k) \\
&\leq \mathcal{F}_j^n(p_{j-1/2}^n \vee k,p_{j+1/2}^n \vee k,p_{j+3/2}^n \vee k) - \mathcal{F}_j^n(p_{j-1/2}^n \wedge k,p_{j+1/2}^n \wedge k,p_{j+3/2}^n \wedge k)\\
&= |p_{j+1/2}^n - k| + \frac{\Delta t}{\Delta x} (\Phi_{j}^n(k) - \Phi_{j+1}^n(k) ).
 \end{align*}
This is exactly the third inequality. Now we treat the case $j = 0$. We have
$$\mathcal{F}_0^n(k_L,k_R,k_R) = k_R - \frac{\Delta t}{\Delta x} \left( H_R(k_R) - F_0(k_L,k_R) \right) $$
Then,
\begin{align*}
&k_R \geq \mathcal{F}_0^n(p_{j-1/2}^n \wedge k_L,p_{j+1/2}^n \wedge k_R,p_{j+3/2}^n \wedge k_R) - \frac{\Delta t}{\Delta x} \left( H_R(k_R) - F_0(k_L,k_R) \right)^- \\
&k_R \leq \mathcal{F}_0^n(p_{j-1/2}^n \vee k_L,p_{j+1/2}^n \vee k_R,p_{j+3/2}^n \vee k_R) + \frac{\Delta t}{\Delta x} \left( H_R(k_R) - F_0(k_L,k_R) \right)^+
\end{align*}
{where $a^\pm=\max(\pm a,0)$},
and we can adapt the previous argument in the following way
\begin{align*}
   |p_{1/2}^{n+1} - k_R | &= p_{1/2}^{n+1}\vee k_R - p_{1/2}^{n+1}\wedge k_R \\
   &= \mathcal{F}_0^n( p_{-1/2}^{n},p_{1/2}^{n},p_{3/2}^{n} ) \vee k_R - \mathcal{F}_0^n( p_{-1/2}^{n},p_{1/2}^{n},p_{3/2}^{n} ) \wedge k_R \\
   &\leq \mathcal{F}_0^n(p_{-1/2}^n \vee k_L,p_{1/2}^n \vee k_R,p_{3/2}^n \vee k_R) + \frac{\Delta t}{\Delta x} \left( H_R(k_R) - F_0(k_L,k_R) \right)^+ \\
   &- \mathcal{F}_0^n(p_{-1/2}^n \wedge k_L,p_{1/2}^n \wedge k_R,p_{3/2}^n \wedge k_R) + \frac{\Delta t}{\Delta x} \left( H_R(k_R) - F_0(k_L,k_R) \right)^- \\
   &= |p_{1/2}^n - k_R| + \frac{\Delta t}{\Delta x} (\Phi_{0}^n(k_\Delta) - \Phi_{1}^n(k_\Delta) + R_R).
\end{align*}
We conclude for the case $j=-1$ with the same procedure. This ends the proof of the lemma.
\end{proof}

We are now ready to prove Lemma \ref{lem:InegEntropiqueDiscrete}.

\begin{proof}[Proof of Lemma \ref{lem:InegEntropiqueDiscrete}]
Let $\phi \in {C}^{\infty}_c([0,T) \times \mathbb{R})$ be non-negative. For all  $j \in \mathbb{Z}$ and $n \in \mathbb{N}$, we define
$$ \phi_{j+1/2}^n := \frac{1}{\Delta x} \int_{x_j}^{x_{j+1}} \phi(t_n,x) \d x.$$
We also denote by $N := \inf\{ n\in \mathbb{N}, t_n > T \}$.
By Lemma \ref{lem:A1} and since $\phi^n_{j+\frac 12}\ge 0$ $\forall n,j$, we have
\begin{align*}
 & \sum_{j \in \mathbb{Z}} \left[ | p_{j+1/2}^{n+1} - k_\Delta| - |p_{j+1/2}^{n} - k_\Delta | \right]\Delta x \; \phi^{n+1}_{j+1/2}
    \\
   \le& - \sum_{j \neq 0,-1} \left[ \Phi^n_{j+1}(k_\Delta) - \Phi^n_j(k_\Delta) \right] \Delta t \; \phi^{n+1}_{j+1/2}
   - \left[ \Phi^n_{1}(k_\Delta) - \Phi^n_0(k_\Delta) - R_R \right] \Delta t \; \phi^{n+1}_{1/2} \\
   &- \left[ \Phi^n_{0}(k_\Delta) - \Phi^n_{-1}(k_\Delta) - R_L \right] \Delta t \; \phi^{n+1}_{-1/2}.
\end{align*}
Using the Abel's transformation and rearranging the terms, we get
\begin{align}
   &\sum_{j \in \mathbb{Z}} \left[ | p_{j+1/2}^{n+1} - k_\Delta| - |p_{j+1/2}^{n} - k_\Delta | \right]\Delta x \; \phi^{n+1}_{j+1/2}\label{eq:302} \\ \nonumber
   &\leq  \left[R_L \phi^{n+1}_{-1/2} + R_R \phi^{n+1}_{1/2}\right] \Delta t +  \sum_{j \in \mathbb{Z}}  \Phi^n_{j}(k_\Delta)\Delta t \; \left[\phi^{n+1}_{j+1/2}-\phi^{n+1}_{j-1/2} \right]
   =: \mathcal I_1 + \mathcal I_2.
\end{align}
First, we estimate $\mathcal I_2$.
\begin{align*}
\mathcal I_2 =& \sum_{j \in \mathbb{Z}}  \Phi^n_{j}(k_\Delta)\Delta t \; \left[\phi^{n+1}_{j+1/2}-\phi^{n+1}_{j-1/2} \right]\\
=& \sum_{j \in \mathbb{Z}}  \Phi^n_{j}(k_\Delta) \frac{\Delta t}{\Delta x} \; \left[\int_{x_{j}}^{x_{j+1}} \phi( t_{n+1},x) \d x - \int_{x_{j-1}}^{x_{j}} \phi(t_{n+1},x) \d x  \right] \\
=& \sum_{j \in \mathbb{Z}}  \Phi^n_{j}(k_\Delta) \frac{\Delta t}{\Delta x} \; \int_{x_{j}}^{x_{j+1}} {\left[\phi( t_{n+1},x) - \phi(t_{n+1},x- \Delta x)\right]} \d x \\
 =& \sum_{j \in \mathbb{Z}}  \Phi^n_{j}(k_\Delta) \frac{\Delta t}{\Delta x} \; \int_{x_{j}}^{x_{j+1}} \left(  \phi_x(t_{n+1},x) \Delta x + \int_{x-\Delta x}^x (x-\Delta x-y)  \phi_{xx}( t_{n+1},y) \d y \right) \d x \\
=&  \int_\mathbb{R}  \Phi_\Delta(k_\Delta)(t_n,x)  \; \Delta t \phi_x(t_{n+1},x) \d x + \sum_{j \in \mathbb{Z}}  \Phi^n_{j}(k_\Delta) \frac{\Delta t}{\Delta x} \int_{x_{j}}^{x_{j+1}} \int_{x-\Delta x}^x (x-\Delta x-y)  \phi_{xx}( t_{n+1},y) \d y \d x \\
 =&  \int_\mathbb{R}  \Phi_\Delta(k_\Delta)(t_n,x)  \;  \int_{t_n}^{t_{n+1}} \left[  \phi_x(t,x) + \int_t^{t_{n+1}} \phi_{tx} (s,x) \d s \right] \d t \d x \\
&+ \sum_{j \in \mathbb{Z}}  \Phi^n_{j}(k_\Delta) \frac{\Delta t}{\Delta x} \int_{x_{j}}^{x_{j+1}} \int_{x-\Delta x}^x (x-\Delta x-y)  \phi_{xx}(t_{n+1},y) \d y \d x \\
=&  \int_\mathbb{R}  \Phi_\Delta(k_\Delta)(t_n,x)  \;  \int_{t_n}^{t_{n+1}}   \phi_x(t,x) \d t \d x + \int_\mathbb{R}  \Phi_\Delta(k_\Delta)(t_n,x) \int_{t_n}^{t_{n+1}} \int_t^{t_{n+1}}  \phi_{tx}(s,x) \d s  \d t \d x \\
&+ \sum_{j \in \mathbb{Z}}  \Phi^n_{j}(k_\Delta) \frac{\Delta t}{\Delta x} \int_{x_{j}}^{x_{j+1}} \int_{x-\Delta x}^x (x-\Delta x-y)  \phi_{xx}(t_{n+1},y) \d y \d x
\end{align*}
Notice that, if we take $(k_L,k_R)\in Q$, then there exists a constant $C$ such that $|\Phi_\Delta| \leq C $. Consequently,
$$\mathcal I_2= \int_\mathbb{R}  \Phi_\Delta(k_\Delta)(t_n,x)  \;  \int_{t_n}^{t_{n+1}}   \phi_x(t,x) \d t \d x + \mathcal I'_2 + \mathcal I_2^{''} $$
where
$$|\mathcal I'_2| \leq C \sup_{t} || \phi_{tx}(t,\cdot)||_{L^1} (\Delta t)^2, \quad |\mathcal I_2^{''}| \leq C \sup_{t} || \phi_{xx}(t,\cdot)||_{L^1} \Delta t \Delta x.$$
We then have
\begin{equation}\label{eq:303}
\mathcal I_2= \int_\mathbb{R}  \Phi_\Delta(k_\Delta)(t_n,x)  \;  \int_{t_n}^{t_{n+1}}   \phi_x(t,x) \d t \d x +O(\Delta t^2)+O(\Delta t \Delta x).
\end{equation}

We now estimate $\mathcal I_1$. Recalling that $R_{F_0}(k_L,k_R) := |H_L(k_L) - F_0(k_L,k_R)| + |H_R(k_R) - F_0(k_L,k_R)|=R_L+R_R$, we have
\begin{align*}
\mathcal I_1 =&\Delta t  \left[R_L \phi^{n+1}_{-1/2} + R_R \phi^{n+1}_{1/2}\right] \\
=& \frac{\Delta t}{\Delta x} \left[R_L \int_{x_{-1}}^{x_0} \phi(t_{n+1}, x) \d x + R_R \int_{x_{0}}^{x_1} \phi(t_{n+1}, x) \d x \right] \\
=& \frac{\Delta t}{\Delta x} \left[(R_L  + R_R )\phi(t_{n+1}, 0)\Delta x + R_L\int_{x_{-1}}^{x_0} \int_0^x  \phi_x(t_{n+1}, y) \d y \d x + R_R \int_{x_{0}}^{x_1} \int_0^x  \phi_x(t_{n+1}, y) \d y \d x \right]\\
=& R_{F_0}(k_L,k_R) \int_{t_n}^{t_{n+1}} \phi(t, 0) \d t + R_{F_0}(k_L,k_R) \int_{t_n}^{t_{n+1}} \int_t^{t_{n+1}}  \phi_t(s,0) \d s \d t \\
&+ \frac{\Delta t}{\Delta x}R_L\int_{x_{-1}}^{x_0} \int_0^x  \phi_x(t_{n+1}, y) \d y \d x + \frac{\Delta t}{\Delta x} R_R \int_{x_{0}}^{x_1} \int_0^x  \phi_x(t_{n+1}, y) \d y \d x \\
=:& R_{F_0}(k_L,k_R) \int_{t_n}^{t_{n+1}} \phi(t, 0) \d t + \mathcal I'_1 + \mathcal I_1^{''} + \mathcal I_1^{'''}
\end{align*}
and there exist a constant $C$ such that
$$|\mathcal I'_1| \leq C || \phi_t ||_\infty (\Delta t)^2, \; \; |\mathcal I_1^{''} + \mathcal I_1^{'''}| \leq  C || \phi_x ||_\infty \Delta t \Delta x.$$
This implies that
\begin{equation}\label{eq:304}
\mathcal I_1=R_{F_0}(k_L,k_R) \int_{t_n}^{t_{n+1}} \phi(t, 0) \d t +O(\Delta t^2)+O(\Delta t \Delta x).
\end{equation}
Combining \eqref{eq:302}, \eqref{eq:303} and \eqref{eq:304}, we finally get
\begin{align*}
&\int_\mathbb{R}  \Phi_\Delta(k_\Delta)(t_n,x)  \;  \int_{t_n}^{t_{n+1}}   \phi_x(t,x) \d t \d x + R_{F_0}(k_L,k_R) \int_{t_n}^{t_{n+1}} \phi(t, 0) \d t + O(\Delta t^2) + O(\Delta x \Delta t)  \\
&\geq \sum_{j \in \mathbb{Z}} \left[ | p_{j+1/2}^{n+1} - k_\Delta| - |p_{j+1/2}^{n} - k _\Delta| \right]\Delta x \; \phi^{n+1}_{j+1/2}
\end{align*}
We sum up with respect to $0 \leq n \leq N$ and use once again Abel's transformation to get
\begin{align*}
&\int_\mathbb{R} \int_0^T   \Phi_\Delta(k_\Delta)(t,x)  \;   \phi_x(t,x) \d t \d x + \int_{0}^{T} R_{F_0}(k_L,k_R)  \phi(t, 0) \d t + O(\Delta t) + O(\Delta x) \\
&\geq \sum_{n=0}^N\sum_{j \in \mathbb{Z}} \left[ | p_{j+1/2}^{n+1} - k_\Delta| - |p_{j+1/2}^{n} - k _\Delta| \right]\Delta x \; \phi^{n+1}_{j+1/2}\\
&\geq \sum_{n=0}^N \sum_{j \in \mathbb{Z}} |p_{j+1/2}^{n} - k_\Delta| \left[ \phi^{n}_{j+1/2} -  \phi^{n+1}_{j+1/2}  \right]\Delta x \;  - \sum_{j \in \mathbb{Z}} |p^0_{j+1/2} - k_\Delta| \phi^0_{j+1/2}\Delta x+\sum_{j\in \mathbb Z} |p^{N+1}_{j+1/2} - k_\Delta| \phi^{N+1}_{j+1/2}\Delta x.
\end{align*}
Recalling that $\phi\in C^\infty_c([0,T)\times \mathbb R)$, we get that
$ \phi^{N+1}_{j+1/2}=0$ for all $j$. Hence
\begin{align*}
&\int_\mathbb{R} \int_0^T   \Phi_\Delta(k_\Delta)(t,x)  \;   \phi_x(t,x) \d t \d x + \int_{0}^{T} R_{F_0}(k_L,k_R)  \phi(t, 0) \d t + O(\Delta t) + O(\Delta x) \\&\geq \sum_{n=0}^N \sum_{j \in \mathbb{Z}} |p_{j+1/2}^{n} - k_\Delta| \int_{x_j}^{x_{j+1}}  \int_{t_n}^{t_{n+1}} - \phi_t(t,x) \d t \d x \;  - \sum_{j \in \mathbb{Z}} |p^0_{j+1/2} - k_\Delta| \int_{x_j}^{x_{j+1}}  \phi(0,x) \d x\\
&\geq -\int_\mathbb{R} \int_0^T |p_\Delta - k_\Delta|   \phi_t(t,x) \d t \d x \;   -\int_\mathbb{R} |p_\Delta(0,x) - k_\Delta| \phi(0,x) \d x
\end{align*}
and we recover the desired discrete entropy inequality.
\end{proof}

\subsection{Local compactness for a numerical scheme of a conservation law}\label{s6b}
The proof of Lemma  \ref{lem:Compactness} is a direct consequence of the following lemma, stated on one branch:
\begin{pro}[{Local} compactness on one branch]\label{pro:compactnessbranch}
Let $f\in C^2(\R)$ be Lipschitz continuous and such that
\begin{equation}\label{eq:assumptionf}
            f'' \geq \delta > 0.
\end{equation}
For $n\ge 0$, we assume that {$q^n_{j+\frac12}$ is given for $j=0$, and for $j\ge 1$ we assume that  $q^{n+1}_{j+\frac 12}$ is} solution of the following scheme
\begin{equation}\label{eq::t2}
q^{n+1}_{j+\frac 12}=q^n_{j+\frac 12}-\frac{\Delta t}{\Delta x}\left(g^f(q_{j+\frac 12}^n, q_{j+\frac 32}^n)-g^f(q^n_{j-\frac 12}, q^n_{j+\frac 12})\right)
\end{equation}
where we recall that the Godunov flux associated to $f$ is given by
 $$g^f(p,q) = \left\{ \begin{matrix} \min_{x\in [p,q]} ( f(x)) &\textrm{ if } p \leq q \\ \max_{x\in [q,p]} ( f(x)) &\textrm{ if } p \geq q. \end{matrix}\right.$$
We assume that $\left|q^n_{j+\frac 12}\right|\le M$ for some $M>0$ {and for all $j,n\ge 0$ and that $\Delta=(\Delta t,\Delta x)$} satisfies
\begin{equation}\label{eq:CFL-f}
\frac{\Delta x}{\Delta t}\ge 2 L_f\qquad{\rm and}\qquad {\gamma:=}\frac{\Delta t}{\Delta x}\frac \delta 2M\le 1
\end{equation}
where $L_f$ is the Lipschitz constant of $f$.
We set
\begin{align*}
   &q_{\Delta} := \sum_{n\in \mathbb{N}} \sum_{j\ge 1} q_{j+1/2}^n \ind_{[t_n, t_{n+1})\times [x_j, x_{j+1})} .
\end{align*}
 Then, there exists $\rho \in L^{\infty}$ and a subsequence also denoted $(q_{\Delta_k})_k$ such that
   $$q_{\Delta_k} \longrightarrow \rho \textrm{ a.e..}$$
\end{pro}

\begin{proof}[Proof of Lemma \ref{lem:Compactness}]
   The proof is a direct consequence of the previous {proposition} applied on $(0,+\infty)$ to $q^{n,+}_{j+\frac 12}=p^n_{j+\frac 12}$ and on $(-\infty, 0)$ to   $q^{n,-}_{j+\frac 12}=p^n_{-j-\frac 12}$ for $j\ge 1$.
\end{proof}

The rest of this section is devoted to the proof of Proposition \ref{pro:compactnessbranch}.
The idea consists to use a localized discrete Oleinik estimate, see Lemma \ref{lem:discreteOleinik}. To prove this estimate, we first need to prove the following discrete ODE on the discrete gradient.

\begin{lemma}[A discrete ODE on the discrete Gradient]\label{lem:BoundDiscreteGradient}
   For $j\ge 1$, let
$$w_j^n := \frac{q^n_{j+1/2} - q^n_{j-1/2}}{\Delta x}$$
and for $j\ge 2$
$$\hat w_j^n := \max \{ 0, w_{j-1}^n, w_{j}^n, w_{j+1}^n \}.$$
Then, for all $j\ge 2$ and for all $n\ge 0$
\begin{equation}\label{eq:boundDGradient}
   \frac{\max(0,w^{n+1}_j) - \hat w_j^n}{\Delta t}  \leq -\frac \delta 8|\hat w^n_j|^2.
\end{equation}
\end{lemma}

\begin{proof}
First, fix $n \in \mathbb{N}$ and $j\ge 2$. we have
\begin{align*}
w_j^{n+1} &=  w_j^n - \frac{\Delta t}{(\Delta x)^2}
   \bigg[ g^{f}(q_{j+1/2}^n ,q_{j+3/2}^n)  - g^{f}(q_{j-1/2}^n,q_{j+1/2}^n) - g^{f}(q_{j-1/2}^n ,q_{j+1/2}^n) + g^{f}(q_{j-3/2}^n , q_{j-1/2}^n )  \bigg]\\
   &=  w_j^n - \frac{\Delta t}{(\Delta x)^2}
   \bigg[ g^{f}(q_{j-1/2}^n + w_j^n \Delta x,q_{j+1/2}^n + w_{j+1}^n \Delta x ) - {2g^{f}(q_{j-1/2}^n,q_{j+1/2}^n)} \\
   &+ g^{f}(q_{j-1/2}^n - w_{j-1}^n\Delta x, q_{j+1/2}^n - w_{j}^n \Delta x)  \bigg]\\
&=:G(w_{j-1}^n, w_j^n, w_{j+1}^n, q_{j-1/2}^n, q_{j+1/2}^n).
\end{align*}
Due to the monotonicity of $g^{f}$, we know that $G$ is non-decreasing with respect to its first and third variables. We now prove that $G$ is also non-decreasing with respect to its second variable. Indeed, we have
\begin{align*}
    \partial_w G(a,w,b,q_{-1},q_{1})
   &= 1 - {\frac{\Delta t}{\Delta x}\left[ \partial_1 g^{f}(q_{-1}+w\Delta x, q_1 + b\Delta x)  - \partial_2 g^{f}(q_{-1} -a \Delta x, q_1-w\Delta x) \right]} \\
   &\geq  1 - 2 \frac{\Delta t}{\Delta x} L_f \geq 0,
\end{align*}
by \eqref{eq:CFL-f}.
This implies that
$$
w_j^{n+1} = G(w_{j-1}^n, w_j^n, w_{j+1}^n, q_{j+1/2}^n, q_{j-1/2}^n) \leq G(\hat w_{j}^n, \hat w_j^n, \hat w_{j}^n, q_{j+1/2}^n, q_{j-1/2}^n).
$$
Moreover,
$$0 = G(0, 0, 0, q_{j+1/2}^n, q_{j-1/2}^n) \leq G(\hat w_{j}^n, \hat w_j^n, \hat w_{j}^n, q_{j+1/2}^n, q_{j-1/2}^n).
$$
This implies that
\begin{equation}\label{eq:200}
   \max(0,w^{n+1}_j) \leq G(\hat w_{j}^n, \hat w_j^n, \hat w_{j}^n, q_{j+1/2}^n, q_{j-1/2}^n).
\end{equation}
For clarity's sake, we omit the $n$ dependency when not necessary.
Set
$$Q_j := \left( \begin{matrix} q_{j-1/2}^n \\ q_{j+1/2}^n \end{matrix} \right), \; \; W_j := \left( \begin{matrix} \hat w_j^n \\ \hat w_j^n \end{matrix} \right).$$
We then get
\begin{equation}\label{eq:204}
\frac{\max(0,w^{n+1}_j) - \hat w_j^n}{\Delta t} \leq - \frac{1}{(\Delta x)^2}
\left[ g^{f}(Q_j+W_j\Delta x) - 2 g^{f}(Q_j) + g^{f}(Q_j - W_j \Delta x) \right].
\end{equation}
We now want to estimate the right {hand} term. Using \eqref{eq:est-gf} in Lemma \ref{lem:diffGodunov} below (with $P=Q_j$, $W=W_j$ and $\alpha=\pm \Delta x$), we have
\begin{equation}\label{eq:201}
\frac{\max(0,w^{n+1}_j) - \hat w_j^n}{\Delta t} \leq -I_j=-(I_j^++I_j^-)
\end{equation}
where for {$\beta=\pm$,
$$I_j^\beta=\int _0^1(1-t)\Hess(g^f)(Q_j+t\beta \Delta x W_j)W_j\cdot W_j dt.$$}
To estimate $I_j^\pm$, we use the explicit form of $\Hess(g^f)(Q_j+t\alpha \Delta x W_j)$ given in Lemma \ref{lem:diffGodunov} below. We assume for the moment that $\hat w^n_j>0$. We then have
$$I_j^+\ge \delta |\hat w^n_j|^2\int _0^1 (1-t)\ind_{\{f^-(q) < f(p), f'(p)> 0 \}}dt$$
where $p=p(t)=q_{j-\frac 12}^n+t\Delta x \hat w^n_j$ and $q=q(t)=q_{j+\frac 12}^n+t\Delta x \hat w^n_j$, and
$$I_j^-\ge \delta |\hat w^n_j|^2\int _0^1 (1-t)\ind_{\{f(q') > f^+(p'), f'(q')< 0  \}}dt$$
where $p'=p'(t)=q_{j-\frac 12}^n-t\Delta x \hat w^n_j$ and $q'=q'(t)=q_{j+\frac 12}^n-t\Delta x \hat w^n_j$. We now want to prove that
\begin{equation}\label{eq:202}
\ind_{\{f^-(q) < f(p), f'(p)> 0 \}}+\ind_{\{f(q') > f^+(p'), f'(q')< 0  \}}\ge 1 \quad \forall t\in ]\frac 12, 1].
\end{equation}
Since $\hat w^n_j>0$, we have $q'-p'= w^n_j \Delta x-2 t \Delta x \; \hat w^n_j\le (1-2t)\Delta x\;  \hat w^n_j<0$ if $t>\frac 12$. Moreover, by definition of $p,q,p',q'$, we have $p'<p$ and $q'<q$.

By {contradiction assume that \eqref{eq:202} is not satisfied, i.e.
$$\left\{\begin{array}{lll}
&f'(p)\le 0 &\quad \mbox{or}\quad f^-(q)\ge f(p)\\
\mbox{and}&&\\
&f(q')\le f^+(p') &\quad \mbox{or}\quad f'(q')\ge 0.\\
\end{array}\right.$$}
 On the one hand, if $f'(p)\le 0$, since $q'<p'{<p}$, we deduce that ${f'(q')<0}$. Hence $f(q')\le f^+(p')$. Since $p'\le p$, we also have $f^+(p')=\inf f$ and so $f(q')=\inf f$ which contradicts the fact that $f'(q')<0$.
On the other hand, if $f'(p)> 0$ and $f^-(q)\ge f(p)$, then $f'(q)<0$.
Since $q'<q$ and $p'<p$, we then get
$$f(q')=f^-(q')>f^-(q)\ge f(p)=f^+(p)\ge f^+(p')$$
which is a contradiction. We then deduce that \eqref{eq:202} holds true. This implies that
$$I_j\ge \delta |\hat w^n_j|^2\int_{1/2}^1(1-t) dt=\frac 18 \delta |\hat w^n_j|^2.$$
Notice that this inequality is also true if $\hat w^n_j=0$.
Injecting this in \eqref{eq:201}, we get the result.
\end{proof}

It remains to show the following lemma concerning some properties of the Godunov flux.

\begin{lemma}[Regularity of the Godunov flux]\label{lem:diffGodunov}
Define
$$\Gamma := \{ (p,q) \textrm{ s.t. } f^+(p) =f^-(q) >\inf_\R f \}.$$
   Then $g^f$ is ${C}^1(\mathbb{R}^2 \backslash \Gamma)$ and
   \begin{align*}
       \nabla g^f (p,q) = \left( \begin{matrix} f'(p) \ind_{\{f^-(q) < f(p), f'(p)> 0 \}} \\ f'(q) \ind_{\{f(q) > f^+(p), f'(q)< 0 \}}. \end{matrix} \right)
   \end{align*}
   Moreover $g^f$ is $W^{2,\infty}(\mathbb{R}^2 \backslash\Gamma)$ and for all $(p,q)\not \in \Gamma$
   $$
      \Hess (g^f ) (p,q) = \left( \begin{matrix} \ind_{\{f^-(q) < f(p), f'(p)> 0 \}} f''(p) &0 \\ 0 &\ind_{\{f(q) > f^+(p), f'(q)< 0 \}} f''(q) \end{matrix} \right).
   $$
      Finally, if $P=(p,q)$ and $W=(w,w)$, then for all $\alpha \in \R$ and for any subgradient $\nabla g^f(P)\in \partial g^f(P)$ (which is a true gradient if $P\not \in \Gamma$)
      \begin{equation}\label{eq:est-gf}
      g^f(P+\alpha W)-g^f(P)\ge \alpha W\cdot \nabla g^f(P)+\alpha^2\int _0^1(1-t)\Hess(g^f)(P+t\alpha W)W\cdot W dt
      \end{equation}
      \end{lemma}
  \begin{proof}
We just prove \eqref{eq:est-gf}, the proof of the other properties being direct consequences of the reformulation of the Godunov flux, in the convex case, $g^f(p,q)=\max(f^+(p), f^-(q))$, given in Lemma \ref{lem:scheme}.

If $w=0$, the result is obvious. Assume that $w\ne 0$. We set $U=[-M,M]^2\backslash \Gamma$. Since $f$ is convex, {$g^f$ is also convex and} we have $D^2g^f\ge \{D^2g^f\}_{|U}\cdot \ind_{U}$, where $ \{D^2g^f\}_{|U}$ is the classical  derivative {part of $D^2g^f$} given by $\Hess(g^f)$. So to prove \eqref{eq:est-gf}, it's sufficient to show
that $\ind_{U}(Q+\alpha t W)=1$ for a.e. $t$. To show this, we claim that for all $t$
$$\Gamma\cap(\Gamma+t W)=\emptyset.$$
Indeed, if there exists $Q=(q_1,q_2)\in \Gamma\cap(\Gamma+t W)$ for some $t\ne 0$ (assume that  $w>0$ and $t>0$ to fix the idea, the other cases being similar), then
$$f^-(q_2+tw)=f^+(q_1+tw)>f^+(q_1)=f^-(q_2)>f^-(q_2+tw)$$
which is a contradiction.
This implies that the curve $t\mapsto Q+\alpha t W$ can cross $\Gamma$ at most one time and so $\ind_{U}(Q+\alpha t W)=1$ for a.e. $t$.
\end{proof}

\begin{lemma}[Discrete Oleinik estimate]\label{lem:discreteOleinik}
Under the same assumptions as Proposition \ref{pro:compactnessbranch}, let $R_2>R_1 > 0$ and {$J_2>J_1 \ge 2$} be such that $(J_1 \Delta x,J_2\Delta x)\subset  (R_1,R_2)$. Then {for $w^n_j$ defined in Lemma \ref{lem:BoundDiscreteGradient} and for $0\le n \leq \frac 12 (J_2-J_1)$, we have}
\begin{equation}\label{eq:OleinikEstimate}
   \frac \delta 8 \sup_{J_1+n \leq j \leq J_2 - n} w_j^n \leq \frac{1}{(n+1) \Delta t}.
\end{equation}
\end{lemma}

\begin{remark}\label{rem::41}
We provide here a proof of the localized estimate (\ref{eq:OleinikEstimate}).
A similar estimate (with possible different constants) can also be deduce from the proofs of the known global results. For Godunov flux, it can be deduced either from  \cite{GL}, or from  \cite{BO} for an optimal constant with a nice proof (which simply uses the fact that Godunov scheme is equivalent to solve exactly the Riemann problem (i.e. solve the exact PDE), and then average the solution). See also \cite{T} for the case of Lax-Friedrichs schemes.
\end{remark}

\begin{proof}[\bf Proof of Lemma \ref{lem:discreteOleinik}]
\noindent {\bf Step 1: Initial condition}\\
We first check that (\ref{eq:OleinikEstimate}) holds true for $n=0$.
We have
{$$w_j^n=\frac{q_{j+1/2}^n-q_{j-1/2}^n}{\Delta x} \quad \mbox{with}\quad |q_{j \pm 1/2}^n|\le M$$}
Hence
$$\left(\Delta t \frac \delta 8\right)  \sup_{j\in [J_1,J_2]} w^0_j \le \left(\Delta t \frac \delta 8\right) \frac{2M}{\Delta x} =\frac{\gamma}{2} \le \frac12 \le 1$$
and (\ref{eq:OleinikEstimate}) is satisfied for $n=0$.\\
\noindent {\bf Step 2: The supersolution}\\
Recall that, by Lemma  \ref{lem:BoundDiscreteGradient}, we have, with $\hat w^n_j:=\max(0,w^n_{j-1},w^n_{j},w^n_{j+1})$, for $j\ge 2$
\begin{equation}\label{eq::37}
\frac{\max(0,w_j^{n+1})-\hat w_j^n}{\Delta t}   \le -\frac \delta 8|\hat w_j^{n}|^2
\end{equation}
Notice that
$$\frac{1}{m+1}-\frac{1}{m} \ge -\left|\frac{1}{m}\right|^2\quad \mbox{for}\quad m\ge 1$$
and then we see immediately that
$$h^n:=\frac{1}{\left(\Delta t \frac \delta 8\right)}\frac{1}{(n+1)}$$
is a supersolution of the equation with equality in (\ref{eq::37}), whose $w^n$ is itself a subsolution.
{Moreover $h^n$ satisfies the equality in the inequality (\ref{eq:OleinikEstimate}) for $n=0$.}\\
\noindent {\bf Step 3: Time evolution and comparison}\\
Now assume that \eqref{eq:OleinikEstimate} is true at  step $n\ge 0$ and let us show it is also true at  step $n+1$.

We then assume that
$$\sup_{j\in [J_1+n,J_2-n]} w^n_j \le h^n$$
i.e.
$$\sup_{j\in [J_1+n+1,J_2-(n+1)]} \hat w^n_j \le h^n.$$
Then (\ref{eq::37}) implies that
$$\sup_{j\in [J_1+(n+1),J_2-(n+1)]} \max(0,w^{n+1}_j) \le \sup_{j\in [J_1+n+1,J_2-(n+1)]} \Phi(\hat w_j^{n})\quad \mbox{with}\quad \Phi(w):=w- \Delta t \frac \delta 8 |w|^2.$$
Because $\Phi$ is nondecreasing on $\left[0,\left( \Delta t \frac \delta 4\right)^{-1}\right]$, and
$$0\le \hat w_j^{n} \le \frac{2 M}{\Delta x} \le \left( \Delta t \frac \delta 4\right)^{-1}{=\frac12 h^0}\quad \mbox{because}\quad \gamma\le 1,$$
we deduce, using that $h^n$ is a supersolution, that
{
$$\Phi(\hat w_j^{n})\le \left\{\begin{array}{lll}
\Phi(h^n) \le h^{n+1}&\quad \mbox{if}\quad n\ge 1,&\quad \mbox{because}\quad h^n\le \left( \Delta t \frac \delta 4\right)^{-1}\\
\Phi(\frac12 h^0)=\frac14 h^0\le \frac12 h^0 = h^1 &\quad \mbox{if}\quad n=0 &\quad \mbox{because}\quad \hat w^0_j \le \frac12 h^0
\end{array}\right.$$}
for all $j\in[J_1+(n+1),J_2-(n+1)]$. This implies that
$$\sup_{j\in [J_1+(n+1),J_2-(n+1)]} \max(0,w^{n+1}_j) \le h^{n+1}.$$
This ends the proof fo the lemma.
\end{proof}
\begin{lemma}\label{lem::38}{\bf (Total variation estimates)}\\
Assume that for $J_2\ge J_1\ge 2$ {and for $B\ge 0$}
$$\left\{\begin{array}{ll}
\displaystyle \frac{q^n_{j+1/2}-q_{j-1/2}^n}{\Delta x} \le B &\quad \mbox{for all}\quad j\in [J_1,J_2-1]\\
\\
|q^n_{j-\frac12}|\le M & \quad \mbox{for all}\quad j\in [J_1,J_2].
\end{array}\right.$$
Then we have
$$\sum_{j\in [J_1,J_2-1]} |q^n_{j+1/2}-q_{j-1/2}^n| \le 2M+ 2B (J_2-J_1) \Delta x$$
and
{$$\sum_{j\in [J_1+1,J_2-1]} |q^{n+1}_{j-1/2}-q_{j-1/2}^n| \le 2L_f \frac{\Delta t}{\Delta x} \sum_{j\in [J_1,J_2-1]} |q^n_{j+1/2}-q_{j-1/2}^n|\le 2L_f \frac{\Delta t}{\Delta x}  \cdot (2M+ 2B (J_2-J_1) \Delta x).$$
 {where} $L_f$ is the Lipschitz constant of $f$.}
\end{lemma}

\begin{proof}The result easily follows from a picture with worse cases (and from the scheme for the last bound). We skip the details.
This ends the proof of the lemma.
\end{proof}

We are now in a position to prove Proposition \ref{pro:compactnessbranch}.

\begin{proof}[Proof of Proposition \ref{pro:compactnessbranch}]
We simply apply the bounds of Lemma \ref{lem::38}, which shows that for all {$\theta> 0$ and $0<R_1<R_2$
$$|q_\Delta|_{BV(\Omega_{\theta,R_1,R_2})} \le C_\theta,\quad |q_\Delta|_{L^\infty(0,+\infty)\times (0,+\infty)} \le M$$
for the triangle
$$\Omega_{\theta,R_1,R_2}:= \left\{(t,x)\in (0,+\infty)^2\; {\rm s.t.} \; t\in \left(\theta, \frac{R_2-R_1}2+\theta\right),\; x\in (R_1+t-\theta, R_2-(t-\theta))\right\}.$$
Recovering $(0,+\infty)\times (0,+\infty)$ by triangles possibly arbitrary small, we deduce the result from a standard diagonal extraction argument.}
This ends the proof of the lemma.
\end{proof}

{
\subsection{Hamilton-Jacobi germs are not $L^1$-dissipative for $N\ge 3$ branches}\label{s6c}

In this subsection, for convenience of an (undeveloped) traffic interpretation/motivation, we prefer to work with concave fluxes instead of convex fluxes (which is indeed equivalent by a simple change of sign).\\

\noindent {\bf Notation}.\\
Let $I$ and $J$ be two non-empty finite sets (of indices) with $I\cap J=\emptyset$. For $\alpha\in I\cup J$, we consider real numbers $a_\alpha< c_\alpha$,  and non constant concave functions $f^\alpha:[a_\alpha,c_\alpha]\to [0,+\infty)$ with $f^\alpha(a_\alpha)=0=f^\alpha(c_\alpha)$.
We consider $\displaystyle A_0:=\min_{\alpha\in I\cup J} \lambda^\alpha_{max}$ where $\displaystyle \lambda^\alpha_{max}:= \max_{Q_\alpha} f^\alpha>0$ and $Q_\alpha:=[a_\alpha,c_\alpha]$.
We set
\begin{equation}\label{defqpmalpha}
f^{\alpha,+}(q)=\sup_{[a_\alpha,q]} f^\alpha,\quad f^{\alpha,-}(q)=\sup_{[q,c_\alpha]} f^\alpha,\quad \mbox{for}\quad q\in Q_\alpha$$
and, for all $\lambda\in [0,\lambda^\alpha_{max}]$,
$$q_\pm^\alpha(\lambda):=q\quad \mbox{where $q\in Q_\alpha$ is defined by}\quad f^\alpha(q)=\lambda=f^{\alpha,\pm}(q)
\end{equation}
We consider weights
\begin{equation}\label{eq::g6}
\theta_\alpha\in (0,1]\quad \mbox{for all}\quad  \alpha\in I\cup J\quad \mbox{such that}\quad  1=\sum_{i\in I} \theta_i=\sum_{j\in J} \theta_j.
\end{equation}
Notice that for $\alpha\in I\cup J$, the equality $\theta_\alpha=1$ implies that $\mbox{Card}(I)=1$ (if $\alpha\in I$) or $\mbox{Card}(J)=1$ (if $\alpha\in J$).\\

\noindent {\bf HJ problem}\\
We consider the following Hamilton-Jacobi problem on a junction with incoming branches indexed by $I$ and outgoing branches indexed by $J$
\begin{equation}\label{eq::g1}
\left\{\begin{array}{rlll}
u^i_t+ \theta_i^{-1}f^i(\theta_iu^i_x)&=0& \quad x<0&\quad  i\in I\\
u^j_t+\theta_j^{-1}f^j(\theta_j u^j_x)&=0& \quad x>0&\quad j\in J\\
u^i=u^j&=:u&\quad x=0&\quad i\in I,\quad j\in J\\
\displaystyle u_t+\min\left\{A,\min_{i\in I} \theta_i^{-1}f^{i,+}(\theta_i u^i_x),\min_{j\in J}\theta_j^{-1}f^{j,-}(\theta_ju^j_x)\right\}&=0& \quad x=0&\\
\end{array}\right.
\end{equation}
where $A\in [0,A_0]$ is the flux limiter.
We define
$\rho^\alpha:=\theta_\alpha u^\alpha_x$ for $\alpha \in I\cup J$,
which satisfies (at least formally)
\begin{equation}\label{eq::g2}
\left\{\begin{array}{rlll}
\rho^i_t+ f^i(\rho^i)_x&=0& \quad x<0&\quad  i\in I\\
\rho^j_t+ f^j(\rho^j)_x&=0& \quad x>0&\quad j\in J\\
\rho=((\rho^i)_{i\in I},(\rho^j)_{j\in J})&\in \mathcal G_A^{HJ}&\quad x=0&\quad  \mbox{for a.e. time $t$}\\
\end{array}\right.
\end{equation}
with the HJ germ defined by the set
$$\mathcal G_A^{HJ}:=\left\{\begin{array}{l}
\displaystyle p=(p_\alpha)_{\alpha\in I\cup J}\in \prod_{i\in I\cup J} Q_\alpha,\quad \mbox{such that there exists $\lambda\in \R$ with}\\
\theta_\alpha^{-1} f^\alpha(p_\alpha)=\lambda=\min\left\{A, \quad \displaystyle \min_{i\in I} \ \theta_i^{-1} f^{i,+}(p_i),\quad  \displaystyle \min_{j\in J} \ \theta_j^{-1} f^{j,-}(p_j)\right\}\quad \mbox{for all}\quad \alpha\in I\cup J
\end{array}\right\}$$
By \eqref{eq::g6} we recover the Rankine-Hugoniot relation
$$\sum_{i\in I} f^i(p_i)=\sum_{j\in J} f^j(p_j)\quad \mbox{for all}\quad p\in \mathcal G_A^{HJ}.$$

\begin{lemma}\label{lem::g5}{\bf (Lack of dissipation for Hamilton-Jacobi germs with $3$ branches or more)}\\
Set $n:=\mbox{Card}(I)$ and $m:=\mbox{Card}(J)$ with $n,m\ge 1$. Under the previous assumptions, we have:\\
\noindent {\bf i)}\ The set $\mathcal G_A^{HJ}$ is   $L^1$-dissipative  if $A\in [0,A_0]$ and $n=m=1$, or if $A=0$ and $n,m\ge 1$.\\
\noindent {\bf ii)}\ For $A\in (0,A_0]$, the set $\mathcal G_A^{HJ}$ is not
$L^1$-dissipative  if $n+m\ge 3$.
\end{lemma}

\begin{proof}[Proof of Lemma \ref{lem::g5}]
Recall that the germ $\mathcal G_A^{HJ}$ is $L^1$-dissipative (on the box $\displaystyle Q:=\prod_{\alpha\in I\cup J} Q_\alpha$) if and only if the entropy flux satisfies $\mbox{IN}\ge \mbox{OUT}$, i.e. for all $p',p\in \mathcal G_A^{HJ}$, we have
\begin{equation}\label{eq::g3}
\sum_{i\in I} \mbox{sign}(p_i'-p_i)\cdot \left\{f^i(p_i')-f^i(p_i)\right\} \ge \sum_{j\in J} \mbox{sign}(p_j'-p_j)\cdot \left\{f^j(p_j')-f^j(p_j)\right\}
\end{equation}
The case $A=0$ is trivial, and we now assume that $A\in (0,A_0]$.
We choose
$$p_i':=q^i_+(\theta_i A),\quad p_j':=q^j_-(\theta_j A),\quad i\in I,\quad j\in J,$$
where the map $q_\pm^\alpha(\cdot)$ is defined in \eqref{defqpmalpha}.
Now we choose $\alpha_0\in I\cup J$ and for some
$\lambda\in (0,A)$, we set
$$
p_i:=\left\{\begin{array}{ll}
q^i_+(\theta_i \lambda) & \text{if}\;  i=\alpha_0\in I,\\
q^i_-(\theta_i \lambda) & \text{if}\;  i\in I\backslash \left\{\alpha_0\right\},
\end{array}\right.
\qquad \text{and}\qquad
p_j:=\left\{\begin{array}{ll} q^j_-(\theta_j \lambda)& \text{if}\;  j=\alpha_0\in J\\
q^i_-(\theta_i \lambda) & \text{if}\;   j\in J\backslash \left\{\alpha_0\right\}
\end{array}\right.
$$
Then we have
$$
\mbox{sign}(p_i'-p_i)=\left\{\begin{array}{ll} +1 & \text{if}\;  i=\alpha_0\in I,\\
-1 & \text{if}\;  i\in I\backslash \left\{\alpha_0\right\},
\end{array}\right.
\qquad \text{and}\qquad
\mbox{sign}(p_j'-p_j)=\left\{\begin{array}{ll} -1& \text{if}\;  j=\alpha_0\in J,\\
+1 & \text{if}\;  j\in J\backslash \left\{\alpha_0\right\}
\end{array}\right.
$$
and
$$\left\{f^\alpha(p_\alpha')-f^\alpha(p_\alpha)\right\} =\theta_\alpha (A-\lambda)>0\quad \mbox{for all}\quad \alpha\in I\cup J.$$
Dividing  (\ref{eq::g3}) by $(A-\lambda)>0$, this leads to:
$$\left\{\begin{array}{lll}
\left\{-1+2\theta_{\alpha_0}\right\}&\ge \left\{+1\right\}&\quad \mbox{if}\quad \alpha_0\in I\\
\left\{-1\right\}&\ge \left\{+1-2\theta_{\alpha_0}\right\}&\quad \mbox{if}\quad \alpha_0\in J\\
\end{array}\right.$$
which forces $\theta_{\alpha_0}\ge 1$. This contradicts (\ref{eq::g6}) if $\mbox{Card}(I)\ge 2$ or $\mbox{Card}(J)\ge 2$.
The fact that $\mathcal G^{HJ}_A$ is $L^1$-dissipative for $\mbox{Card}(I)=1=\mbox{Card}(J)$ is proved in Proposition \ref{pro:propertiesGAF0}. This ends the proof of the lemma.
\end{proof}
}

\paragraph{\textbf{Acknowledgement.}}
This research was partially funded by l'Agence Nationale de la Recherche (ANR), project ANR-22-CE40-0010 COSS.
For the purpose of open access, the authors have applied a CC-BY public copyright licence to any Author Accepted Manuscript (AAM) version arising from this submission. The last author would like to thank J. Dolbeault, C. Imbert and T. Leli\`{e}vre for providing him good working conditions.

\bibliographystyle{siam}

\begin{thebibliography}{}

\end{thebibliography}


\begin{thebibliography}{10}

\bibitem{ACCT}
{\sc Y.~Achdou, F.~Camilli, A.~Cutr\`{i}, and N.~Tchou}, {\em Hamilton-{J}acobi
  equations constrained on networks}, NoDEA Nonlinear Differential Equations
  Appl., 20 (2013), pp.~413--445.

\bibitem{AdimurthiGowda2005}
{\sc Adimurthi, S.~Mishra, and G.~D.~V. Gowda}, {\em Optimal entropy solutions
  for conservation laws with discontinuous flux-functions}, J. Hyperbolic
  Differ. Equ., 2 (2005), pp.~783--837.

\bibitem{AGS2010}
{\sc B.~Andreianov, P.~Goatin, and N.~Seguin}, {\em Finite volume schemes for
  locally constrained conservation laws}, Numer. Math. (Heidelb.), 115 (2010),
  pp.~609--645.

\bibitem{andreianov}
{\sc B.~Andreianov, K.~H. Karlsen, and N.~H. Risebro}, {\em A theory of
  {$L^1$}-dissipative solvers for scalar conservation laws with discontinuous
  flux}, Arch. Ration. Mech. Anal., 201 (2011), pp.~27--86.

\bibitem{AbrahamPreprint}
{\sc B.~Andreianov and A.~Sylla}, {\em Finite volume approximation and
  well-posedness of conservation laws with moving interfaces under abstract
  coupling conditions}.
\newblock  {Nonlinear Differ. Equ. Appl. 30, 53 (2023)}.


\bibitem{CancesAndreianov2015}
{\sc B.~P. Andreianov and C.~Canc{\`e}s}, {\em On interface transmission
  conditions for conservation laws with discontinuous flux of general shape},
  Journal of Hyperbolic Differential Equations, 12 (2015), pp.~343--384.

\bibitem{BCD2016}
{\sc B.~P. Andreianov, G.~M. Coclite, and C.~Donadello}, {\em Well-posedness
  for vanishing viscosity solutions of scalar conservation laws on a network},
  Discrete and Continuous Dynamical Systems, 37 (2017), pp.~5913--5942.

\bibitem{MonotoneGraphAndreianov}
{\sc {Andreianov, Boris}}, {\em New approaches to describing admissibility of
  solutions of scalar conservation laws with discontinuous flux}, ESAIM:
  Proc., 50 (2015), pp.~40--65.

\bibitem{BLN}
{\sc C.~Bardos, A.~Y. le~Roux, and J.-C. N\'{e}d\'{e}lec}, {\em First order
  quasilinear equations with boundary conditions}, Comm. Partial Differential
  Equations, 4 (1979), pp.~1017--1034.

\bibitem{BCbook}
{\sc G.~Barles and E.~Chasseigne}, {\em An illustrated guide of the modern
  approaches of hamilton-jacobi equations and control problems with
  discontinuities}. {https://arxiv.org/abs/1812.09197}, 2023.

\bibitem{BO}
{\sc Y.~Brenier and S.~Osher}, {\em The discrete one-sided {L}ipschitz
  condition for convex scalar conservation laws}, SIAM J. Numer. Anal., 25
  (1988), pp.~8--23.


\bibitem{Caselles92}
{\sc V.~Caselles}, {\em Scalar conservation laws and {H}amilton-{J}acobi
  equations in one-space variable}, Nonlinear Anal., 18 (1992), pp.~461--469.

\bibitem{ColomboGoatin2007}
{\sc R.~M. Colombo and P.~Goatin}, {\em A well posed conservation law with a
  variable unilateral constraint}, J. Differ. Equ., 234 (2007), pp.~654--675.

  \bibitem{CP20}
\textsc{R.~M. Colombo, and V.~Perrollaz},
{\it Initial data identification in conservation laws and Hamilton-Jacobi equations}.
Journal de Math\'ematiques Pures et Appliqu\'ees, 138 (2020), 1-27.


\bibitem{colombo-Perrollaz-Sylla}
{\sc R.~M. Colombo, V.~Perrollaz, and A.~Sylla}, {\em {Initial Data
  Identification in Space Dependent Conservation Laws and Hamilton-Jacobi
  Equations}}.
\newblock {https://hal.science/hal-04062783/,} 2023.

\bibitem{CR2015}
{\sc R.~M. Colombo and E. Rossi}, {\em {Rigorous estimates on balance laws in bounded domains}}.
Acta Math. Sci. Ser. B (Engl. Ed.), 35 (4), (2015), pp.906--944.


\bibitem{coste}
{\sc M.~Coste}, {\em An introduction to o-minimal geometry}.
\newblock RAAG Notes, 81 pages, Institut
de Recherche Mathematiques de Rennes, November 1999. Lectures notes,
  https://perso.univ-rennes1.fr/michel.coste/polyens/OMIN.pdf.

\bibitem{coste2}
{\sc M.~Coste}, {\em An Introduction to Semialgebraic Geometry.}
\newblock RAAG Notes, 78 pages, Institut de Recherche Mathï¿½matiques de Rennes, October 2002.

\bibitem{FMR22}
{\sc U.~S. Fjordholm, M.~Musch, and N.~H. Risebro}, {\em Well-posedness and
  convergence of a finite volume method for conservation laws on networks},
  SIAM J. Numer. Anal., 60 (2022), pp.~606--630.

\bibitem{GL}
{\sc J.~B. Goodman and R.~J. LeVeque}, {\em A geometric approach to high
  resolution {TVD} schemes}, SIAM J. Numer. Anal., 25 (1988), pp.~268--284.

\bibitem{guerand}
{\sc J.~Guerand and M.~Koumaiha}, {\em Error estimates for a finite difference
  scheme associated with {H}amilton-{J}acobi equations on a junction}, Numer.
  Math., 142 (2019), pp.~525--575.

\bibitem{KR02}
\textsc{K.H. Karlsen and  N.H. Risebro},
{\it A note on Front tracking and the Equivalence between Viscosity Solutions of Hamilton-Jacobi
Equations And Entropy Solutions of scalar Conservation Laws.}
Nonlin. Anal. TMA 50 (4) (2002), 455-469.


\bibitem{imbert-monneau}
{\sc C.~Imbert and R.~Monneau}, {\em Flux-limited solutions for quasi-convex
  {H}amilton-{J}acobi equations on networks}, Ann. Sci. \'{E}c. Norm.
  Sup\'{e}r. (4), 50 (2017), pp.~357--448.

\bibitem{IMZ}
{\sc C.~Imbert, R.~Monneau, and H.~Zidani}, {\em A {H}amilton-{J}acobi approach
  to junction problems and application to traffic flows}, ESAIM Control Optim.
  Calc. Var., 19 (2013), pp.~129--166.

\bibitem{LS1}
{\sc P.-L. Lions and P.~Souganidis}, {\em Viscosity solutions for junctions:
  well posedness and stability}, Atti Accad. Naz. Lincei Rend. Lincei Mat.
  Appl., 27 (2016), pp.~535--545.

\bibitem{LS2}
\leavevmode\vrule height 2pt depth -1.6pt width 23pt, {\em Well-posedness for
  multi-dimensional junction problems with {K}irchoff-type conditions}, Atti
  Accad. Naz. Lincei Rend. Lincei Mat. Appl., 28 (2017), pp.~807--816.

\bibitem{MISHRASurvey}
{\sc S.~Mishra}, {\em Chapter 18 - numerical methods for conservation laws with
  discontinuous coefficients}, in Handbook of Numerical Methods for Hyperbolic
  Problems, R.~Abgrall and C.-W. Shu, eds., vol.~18 of Handbook of Numerical
  Analysis, Elsevier, 2017, pp.~479--506.

  \bibitem{Monneau2023}
{\sc R.~Monneau}, {\rm Strictly convex Hamilton-Jacobi equations: strong trace of the gradient}. Preprint, https://hal.science/hal-04254243.

\bibitem{MFR22}
{\sc M.~Musch, U.~S. Fjordholm, and N.~H. Risebro}, {\em Well-posedness theory
  for nonlinear scalar conservation laws on networks}, Netw. Heterog. Media, 17
  (2022), pp.~101--128.

\bibitem{panov-trace}
{\sc E.~Y. Panov}, {\em Existence of strong traces for quasi-solutions of
  multidimensional conservation laws}, J. Hyperbolic Differ. Equ., 4 (2007),
  pp.~729--770.

\bibitem{Abraham2022}
{\sc A.~Sylla}, {\em A lwr model with constraints at moving interfaces}, ESAIM:
  Mathematical Modelling and Numerical Analysis, 56 (2022).

\bibitem{T}
{\sc E.~Tadmor}, {\em The large-time behavior of the scalar, genuinely
  nonlinear {L}ax-{F}riedrichs scheme}, Math. Comp., 43 (1984), pp.~353--368.

\bibitem{Trelat1}
{\sc E.~Trélat}, {\em Global subanalytic solutions of Hamilton-Jacobi type equations}, Ann. Inst. H. Poincaré C Anal. Non Linéaire 23 (2006), no. 3, 363--387.

\bibitem{Trelat2}
{\sc E.~Trélat}, {\em Solutions sous-analytiques globales de certaines équations d'{H}amilton-{J}acobi}, Comptes Rendus Math. 337, 10 (2003), 653--656.




\end{thebibliography}

\end{document}